\documentclass[a4paper,10pt,twoside]{article}

\usepackage{latexsym,amsmath,amssymb}\usepackage{amsthm,amscd,mathrsfs}
\usepackage{bm}
\usepackage[pdftex]{graphicx}
\usepackage{pgfplots}
\usepackage{tikz}
\usetikzlibrary{positioning}
\usetikzlibrary{arrows}
\usepackage{fancybox}
\usepackage{multicol}
\usepackage{tabularx}
\usepackage{comment}
\usepackage{cases}
\usepackage[top=3.5cm,bottom=3.5cm,left=3cm,right=3cm]{geometry}

\newtheorem{thm}{Theorem}[section]

\newtheorem{lem}{Lemma}[section]

\newtheorem{df}{Definition}[section]

\theoremstyle{definition}

\newtheorem{rem}{Remark}[section]
\newtheorem{ex}{Example}[section]
\newtheorem{prob}{Problem}[section]
\newtheorem{sol}{Solution\!\!}

\makeatletter
 
 \@addtoreset{equation}{section}
\makeatother

\title{{\sc Type-2 Fuzzy Initial Value Problems for Second-order T2FDEs}}
\author{N. Someyama$^{*1}$, H. Uesu$^{*2}$, K. Shinkai$^{*3}$ $\&$ S. Kanagawa$^{*4}$}
\date{{\small 
$^{*1}$Shin-yo-ji Buddhist Temple, 5-44-4 Minamisenju, Arakawa-ku, Tokyo 116-0003, Japan\\
{\tt philomatics@outlook.jp}\\ \vspace{3mm}
$^{*2}$Kanazawa Institute of Technology, 7-1 Ohgigaoka, Nonoichi-shi, Ishikawa 921-8501, Japan\\
{\tt uesu@neptune.kanazawa-it.ac.jp}\\ \vspace{3mm}
$^{*3}$Tokyo Kasei Gakuin University, 2600 Aihara, Machida-shi, Tokyo 194-0211, Japan
\\
{\tt k-shinkai@kasei-gakuin.ac.jp}\\ \vspace{3mm}
$^{*4}$Tokyo Gakugei University, 4-1-1 Nukuikita-machi, Koganei-shi, Tokyo 184-8501, Japan
\\
{\tt sgk02122@nifty.ne.jp}
}}
\begin{document}
\maketitle
\markboth{SUSK}{Type-2 Fuzzy Initial Value Problems for Second-order T2FDEs}

\begin{abstract}
Type-2 fuzzy differential equations (T2FDEs) of order 1 are already known and the solution method of type-2 fuzzy initial value problems (T2FIVPs) for them was given by M. Mazandarani and M. Najariyan \cite{MN} in 2014.
We give the solution method of second-order T2FIVPs in this paper.
Furthermore, we would like to propose new notations for type-2 fuzzy theory where symbols tend to be complicated and misleading.
In particular, the Hukuhara differential symbols introduced experimentally in this paper will give us clearler meanings and expressions.
\end{abstract}
\vspace{3mm}

\vspace{2mm}

\vspace{2mm}

{\small
{\bf Corresponding Author}: Norihiro Someyama $<${\tt philomatics@outlook.jp}$>$
}

{\small
\tableofcontents
}

\section{Introduction}
In 1965, fuzzy set theory \cite{Z1} was introduced as the origin of mathematical theory of ambiguity by L.A. Zadeh (1921-2017).
A fuzzy set $A$ on the universal set $X$ is characterized via the membership function $\mu_A:X\to [0,1]$ and the membership function value $\mu_A(x)$ which represents the grade of ambiguity for each $x\in X$.
Any membership function however is basically formed from its individual function values determined by the subjectivity of the observer.
So, the grades $\{\mu_A(x):x\in X\}$ may also contain ambiguity.
Focusing on this, Zadeh \cite{Z2} introduced type-2 fuzzy set theory in 1975.
Type-2 fuzzy sets are fuzzy sets whose grades are fuzzy.
In other words, the type-2 fuzzy set includes not only the uncertainty of the data, but also the membership function which indicates the uncertainty.
Then, we consider the membership function of a type-2 fuzzy set ${\cal A}$ as $\mu_{{\cal A}}:X\to [0,1]^{[0,1]}$.

The concept of fuzzy derivatives was proposed by Chang $\&$ Zadeh \cite{CZ}.
In addition, fuzzy derivatives using the extension principle were proposed by Dubois $\&$ Prade \cite{DP}, and some other concepts related to fuzzy derivatives were discussed by Puri $\&$ Ralescu \cite{PR1}.
Fuzzy initial value problems have been researched since Kaleva \cite{K} and Bede $\&$ Gal \cite{BG}, and several attempts have been proposed to define the differentiability of fuzzy functions.
Among them, Hukuhara differentiability and strongly generalized differentiability \cite{BG,PR1} have attracted particular attention.
For the sake of simplicity, `type-1 / type-2 fuzzy differential equations' and `type-1 / type-2 fuzzy initial value problems' are often abbreviated as `T1/T2FDEs' and `T1/T2FIVPs' respectively in this paper.

For example, suppose that we have a highly experienced expert and an inexperienced student measure the temperature of a certain substance and ask them to indicate the membership function of temperature (More specifically, see Section 5 of \cite{MN}).
In this case, there is a possibility that different membership functions will be expressed, and the former and the latter can be recognaized as the principle set and the foot-print set (see Definition \ref{df:FP}-\ref{df:PS}), respectively.
Therefore, type-2 fuzzy sets are useful if the exact form of the membership function is not known, or if the grade of the membership function itself is ambiguous or inaccurate.
Since there are so many problems where the exact form of the membership function cannot be determined, type-2 fuzzy sets are suitable for dealing with high levels of uncertainty that involve more complicated calculations.
Moreover, the parameters and variables appearing in differential equations in real problems are usually very imprecise, but we may well be able to model them by type-2 fuzzy theory and hence T2FDEs.

In 2014, Mazandarani $\&$ Najariyan \cite{MN} studied first-order T2FDEs and took up some concrete T2FIVPs for them.
Related to this, we study, in this paper, second order T2FDEs and T2FIVPs for them.
We also present and prove other theorems that may not be known in the case of type-1 and type-2 first order.
We attack the case of crisp coefficients in this paper.
The physical phenomena of our world are generally represented by second-order differential equations, so we believe that our study will be necessary and useful in fields such as mathematical physics.
In fact, we can think that it is appropriate to set initial values as fuzzy numbers if an expert measures the values.
Furthermore, although the accuracy of experiments is improving remarkably in science, it is useful to discuss the accuracy of old experiments and that of present experiments as the foot-print set and the principle set, respectively.

Incidentally, it may be necessary to seek the best notation because type-2 fuzzy theory tends to be complicated in notations.
We would like to propose some new notations on a trial basis in this paper.

\section{Preparation}
We prepare definitions of terms, notations and known results on type-2 fuzzy theory required in this paper.
The knowledges on type-1 fuzzy theory will be also required, but we put them in the appendix because writing them in this section makes them redundant.

For convenience, we use a notation such as `the fuzzy set $A:X\to [0,1]$' and basically write $A(x)$ for the grade of $A$.
`Crisp' expresses `non-fuzzy'.

\subsection{Type-2 Fuzzy Numbers}
We first introduce type-2 fuzzy sets.
A type-2 fuzzy set is defined by its membeship function with a fixed input order (See Remark \ref{rem:psmf}, 1), later for details).

\begin{df}[\cite{ML}]
\label{df:T2FS}
A type-2 fuzzy set ${\cal A}$ on $X$ is characterized by
\begin{align}
\label{eq:pvrep}
{\cal A}:=\{(x,u;\nu_{{\cal A}}(x,u)):x\in X,u\in R(\mu_{{\cal A}}(x))\subset [0,1]\}
\end{align}
where 
\begin{itemize}
\item $\nu_{{\cal A}}$ is the membership function of ${\cal A}$ from the ordered pair $(x,u)$ to $[0,1]$,
\item $R(\mu_{{\cal A}}(x))$, for each $x\in X$, is the range of $\mu_{{\cal A}}:X\to [0,1]^{[0,1]}$ called the primary membership function of ${\cal A}$,
\item $x$ is called the primary variable of ${\cal A}$,
\item $u$ is called the secondary variable of ${\cal A}$.
\end{itemize}
\end{df}

{\small
\begin{rem}
\label{rem:psmf}
\quad
\begin{itemize}
\item[1)] $\nu_{{\cal A}}$ is not just a two-variable function $X\times R(\mu_{{\cal A}})\to [0,1]$.
The conventional two-variable function has no restriction on the order of inputting variables, but in the case of $\nu_{{\cal A}}$, $x$ is input and then $u$ must be input.
However, there is another definition of type-2 fuzzy sets.
See Conclusion of this paper.
\item[2)] (\ref{eq:pvrep}) is often called the point-valued representation of ${\cal A}$.
\item[3)] The above $\mu_{\cal A}(x)$ (resp. $u$) should be recognized as the fuzzy grade (resp. the grade of the fuzzy grade) of ${\cal A}$ at $x\in X$.
\end{itemize}
\end{rem}
}

\begin{df}[\cite{ZM}]
Let ${\cal A}$ be a type-2 fuzzy set on $X$ and $x_0\in X$ a fixed point.
We define the type-1 fuzzy set 
\begin{align*}
\kappa_{{\cal A}}(x_0):=\int_{u\in R(\mu_{{\cal A}}(x_0))}\nu_{{\cal A}}^{x_0}(u)/u
\end{align*}
where $\nu_{{\cal A}}^{x_0}:=\nu_{{\cal A}}(x_0,\,\cdot\,):R(\mu_{{\cal A}}(x_0))\to [0,1]$ is called the secondary membership function of ${\cal A}$ at $x_0$.
Moreover, $\nu_{{\cal A}}(x_0,u)$ is called the secondary grade of $x_0$.
Here, similar to the well-known notation for type-1 fuzzy sets, the above integral symbol does not mean the conventional continuous sum, but just a continuous union.
\end{df}

\begin{df}[\cite{H}, Definition 2.8.1]
Let ${\cal A}$ be a type-2 fuzzy set on $X$.
The $\beta$-cut set of ${\cal A}$ is defined by
\begin{align*}
[{\cal A}]_{\beta}:=\left\langle \underline{A}_{\beta},\ \overline{A}_{\beta} \right\rangle
:=\int_{x\in X}\int_{u\in R(\mu_{{\cal A}})}\{(x,u):\nu_{{\cal A}}^x(u)\ge \beta\}
\end{align*}
for any $\beta \in [0,1]$.
Then, $\underline{A}_{\beta}$ and $\overline{A}_{\beta}$ are called the lower membership function and the upper membership function of ${\cal A}$ respectively.
Moreover, the $\alpha$-cut set of $[{\cal A}]_{\beta}$ is defined by
\begin{align}
\label{eq:[A]ba}
[{\cal A}]_{\beta}^{\alpha}:=\left\langle [\underline{A}_{\beta}]_{\alpha},\ [\overline{A}_{\beta}]_{\alpha} \right\rangle
\end{align}
for any $\alpha \in [0,1]$, where $\underline{A}_{\beta}$ and $\overline{A}_{\beta}$ are type-1 fuzzy numbers on $X$ that appear if ${\cal A}$ is cut by $\beta$.
\end{df}

{\small
\begin{rem}
The $\beta$-cut set of a type-2 fuzzy set is also called the $\beta$-plane of it.
Strictly speaking, the $\alpha$-cut set of the $\beta$-cut set of a type-2 fuzzy set ${\cal A}$ should be represented as $[[{\cal A}]_{\beta}]_{\alpha}$, but we will write it like (\ref{eq:[A]ba}) since that is annoying.
\end{rem}
}

The argument of type-2 fuzzy sets can be reduced to that of the level cut sets as with type-1 sets.
In fact, it is known \cite{H} that the level cut sets make up the original type-2 fuzzy set ${\cal A}$:
\begin{align*}
{\cal A}=\int_{x\in X}\left(\int_{u\in R(\mu_{{\cal A}}(x))}\nu_{{\cal A}}^x(u)/u\right)\biggr/ x
=\bigcup_{\beta\in [0,1]}\beta[{\cal A}]_{\beta}
=\bigcup_{\beta\in [0,1]}\beta \bigcup_{\alpha\in [0,1]}\alpha[{\cal A}]_{\beta}^{\alpha}
\end{align*}
where $\alpha[{\cal A}]_{\beta}^{\alpha}:X\to \{0,\alpha\}$ is a type-1 fuzzy set.
Thus, it is sufficient to consider and argue $\beta$-cut sets or these $\alpha$-cut sets for most problems.
We hereinafter write $S_{{\cal A}}(x_0;\beta)$ for the $\beta$-cut set of the secondary membership function $\nu_{{\cal A}}^{x_0}$ of ${\cal A}$ and do the same for ${\cal B}$.
We hereafter omit the description `$\alpha\in [0,1]$' and `$\beta\in [0,1]$' when we argue $(\alpha,\beta)$-cut sets.

\begin{df}
Let ${\cal A}$ and ${\cal B}$ be type-2 fuzzy sets on $X$.
We denote the $\beta$-cut sets of them by
\begin{align}
\label{eq:ABbetac}
[{\cal A}]_{\beta}:=\left\langle \underline{A}_{\beta},\ \overline{A}_{\beta} \right\rangle,\quad
[{\cal B}]_{\beta}:=\left\langle \underline{B}_{\beta},\ \overline{B}_{\beta} \right\rangle.
\end{align}
Then, ${\cal A}={\cal B}$ if and only if $S_{{\cal A}}(x;\beta)=S_{{\cal B}}(x;\beta)$ for all $x\in X$ and any $\beta\in [0,1]$.
\end{df}

\begin{df}[\cite{H}]
\label{df:+kop}
Let ${\cal A}$ and ${\cal B}$ be type-2 fuzzy sets on $X$ and $k\in \mathbb R$.
We denote the $\beta$-cut sets of ${\cal A}$ and ${\cal B}$ by (\ref{eq:ABbetac}).
The sum ${\cal A}+{\cal B}$ of ${\cal A}$ and ${\cal B}$ is defined by
\begin{align*}
[{\cal A}+{\cal B}]_{\beta}:=\left\langle [\underline{A}_{\beta}+\underline{B}_{\beta}]_{\alpha},\ [\overline{A}_{\beta}+\overline{B}_{\beta}]_{\alpha}\right\rangle =\left\langle [\underline{A}_{\beta}]_{\alpha}+[\underline{B}_{\beta}]_{\alpha},\ [\overline{A}_{\beta}]_{\alpha}+[\overline{B}_{\beta}]_{\alpha}\right\rangle.
\end{align*}
Moreover, the scalar multiple $k{\cal A}$ of ${\cal A}$ is defined by
\begin{align*}
[k{\cal A}]_{\beta}:=\left\langle [k\underline{A}_{\beta}]_{\alpha},\ [k\overline{A}_{\beta}]_{\alpha}\right\rangle =\left\langle k[\underline{A}_{\beta}]_{\alpha},\ k[\overline{A}_{\beta}]_{\alpha}\right\rangle.
\end{align*}
\end{df}

\begin{df}
Let ${\cal A},{\cal B}$ be type-2 fuzzy sets on $X$.
We denote the $\beta$-cut sets of them by (\ref{eq:ABbetac}).
Then, the order relationship between them, ${\cal A}\le {\cal B}$, is defined as
\begin{align*}
[{\cal A}]_{\beta}\le [{\cal B}]_{\beta}
\end{align*}
for any $\beta \in [0,1]$.
In particular, the non-negativity (resp. positivity) of a type-2 fuzzy set ${\cal A}$, ${\cal A}\ge 0$ (resp. ${\cal A}>0$), is defined by
\begin{align*}
[\underline{A}_{\beta}]_{\alpha}\ge 0\ {\rm and}\ [\overline{A}_{\beta}]_{\alpha}\ge 0\quad 
({\rm resp.}\ [\underline{A}_{\beta}]_{\alpha}>0\ {\rm and}\ [\overline{A}_{\beta}]_{\alpha}>0)
\end{align*}
for any $\alpha,\beta \in [0,1]$.
\end{df}

The following concepts are important to consider the type-2 version of a type-1 triangular fuzzy number and it is effective when we actually solve concrete T2FIVPs in Section \ref{sec:ExT2FDE}.

\begin{df}[\cite{Me}]
\label{df:FP}
Let ${\cal A}$ be a type-2 fuzzy set on $X$.
The union of all secondary domains
\begin{align*}
{\rm FP}({\cal A}):=\bigcup_{x\in X}R(\mu_{{\cal A}}(x))
\end{align*}
of ${\cal A}$ is called the foot-print set (or foot-print of uncertainty) of ${\cal A}$.
\end{df}

\begin{df}[\cite{H}, Definition 2.3.9]
\label{df:PS}
Let ${\cal A}$ be a type-2 fuzzy set on $X$.
Suppose that there exists at least one $u\in R(\mu_{{\cal A}}(x))$ satisfying
\begin{align*}
\nu_{{\cal A}}^x(u)=\nu_{{\cal A}}(x,u)=1
\end{align*}
for any $x\in X$.
If we rewrite $u_x$ for each such point $u\in R(\mu_{{\cal A}}(x))$, every $u_x$ is equal to the membership function value of the type-1 fuzzy set which is uniquely determined.
Then, that type-1 fuzzy set is called the principle set of ${\cal A}$ and is denoted by ${\rm P}({\cal A})$.
\end{df}

\begin{figure}[h]
\begin{center}
\includegraphics[width=12cm]{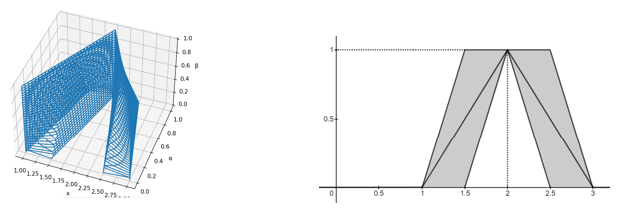}
\caption{Foot-print set}
\end{center}
\end{figure}

\begin{figure}[h]
\begin{center}
\includegraphics[width=12cm]{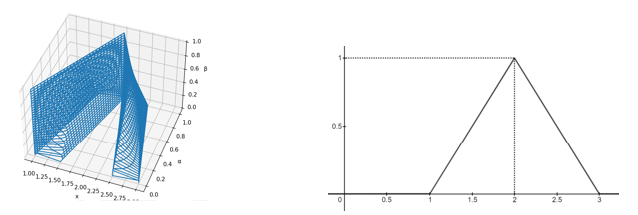}
\caption{Principle set}
\end{center}
\end{figure}

\begin{df}[\cite{H}, Section 3.4]
Let ${\cal A}\in \mathscr{T}^2(\mathbb R)$.
${\cal A}$ is perfect if and only if
\begin{itemize}
\item[i)] the upper and lower membership functions of ${\rm FP}({\cal A})$ are equal as type-1 fuzzy numbers, and
\item[ii)] the upper and lower membership functions of ${\rm P}({\cal A})$ are equal as type-1 fuzzy numbers.
\end{itemize}
Moreover, if a perfect ${\cal A}$ also satisfies that
\begin{itemize}
\item[iii)] ${\cal A}$ can be completely determined by using its ${\rm FP}({\cal A})$ and ${\rm P}({\cal A})$,
\end{itemize}
such a ${\cal A}$ is called the perfect quasi-type-2 fuzzy number on $\mathbb R$.
The space of them is denoted by $\mathscr{QT}^2(\mathbb R)$.
\end{df}

{\small
\begin{ex}
Consider ${\cal A}\in \mathscr{T}^2(\mathbb R)$ such that
\begin{itemize}
\item Primary: $\mu_{{\cal A}}(x)=\max\{1-|x-2|,\ 0\}$,
\item Secondary: $\begin{array}{rl}\nu_{{\cal A}}^{x_0}(u)&\hspace{-2.5mm}=\max\{1-10|u-x_0|,\ 0\} \\
&\hspace{-2.5mm}=\max\Bigl\{1-10\Bigl|u-\max\{1-|x-2|,0\}\Bigr|,\ 0\Bigr\}\quad (0\le u\le 1).\end{array}$
\end{itemize}
Then, the lower membership function of ${\rm FP}({\cal A})$ is given by
\begin{align*}
\max\Bigl\{1-10\Bigl|u-\max\{1-|x-2|,0\}\Bigr|,\ 0\Bigr\}=0,
\end{align*}
whereas, for e.g. $x\in [1,3]$, solving
\begin{align*}
\max\Bigl\{1-10\Bigl|u-1+|x-2|\Bigr|,\ 0\Bigr\}=0
\end{align*}
implies
\begin{align*}
u=1-|x-2|\pm \frac{1}{10}.
\end{align*}
The lower membership function of ${\rm FP}({\cal A})$ in this case, thus, is given by
\begin{align*}
u=\frac{9}{10}-|x-2|<1.
\end{align*}
Hence, this is not normal, so $u\notin \mathscr{T}^1$.
This thing implies that ${\cal A}$ is {\bf not} perfect.
\end{ex}
}

\begin{figure}[h]
\begin{center}
\includegraphics[width=6cm]{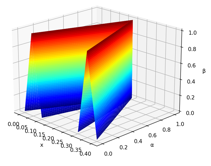}
\caption{Perfect quasi-type-2 fuzzy number}
\end{center}
\end{figure}

It is pointed out in \cite{MN} that the `triangular fuzzy number' in the type-2 world is characterized as follows:

${\cal A}\in \mathscr{QT}^2(\mathbb R)$ is triangular if and only if $[{\cal A}]_{\beta}^{\alpha}$ has
\begin{align}
[\underline{A}_{\beta}]_{\alpha}&=\left[L_{\underline{A}_{\beta}}^{\alpha},\ R_{\underline{A}_{\beta}}^{\alpha}\right], \label{eq:uAba}\\
L_{\underline{A}_{\beta}}^{\alpha}&=X_{A_{1}}^{\alpha}-(1-\beta)\left(X_{A_{1}}^{\alpha}-L_{\underline{A}_{0}}^{\alpha}\right), \label{eq:LuAba}\\
R_{\underline{A}_{\beta}}^{\alpha}&=Y_{A_{1}}^{\alpha}+(1-\beta)\left(R_{\underline{A}_{0}}^{\alpha}-Y_{A_{1}}^{\alpha}\right), \label{eq:RuAba}\\
L_{\underline{A}_{0}}^{\alpha}&=C_{{\cal A}}-(1-\alpha)\left(C_{{\cal A}}-L_{\underline{A}_{0}}\right), \label{eq:LuA0a}\\
R_{\underline{A}_{0}}^{\alpha}&=C_{{\cal A}}+(1-\alpha)\left(R_{\underline{A}_{0}}-C_{{\cal A}}\right)
\label{eq:RuA0a}
\end{align}
and
\begin{align}
[\overline{A}_{\beta}]_{\alpha}&=\left[L_{\overline{A}_{\beta}}^{\alpha},\ R_{\overline{A}_{\beta}}^{\alpha}\right], \label{eq:oAba}\\
L_{\overline{A}_{\beta}}^{\alpha}&=X_{A_{1}}^{\alpha}-(1-\beta)\left(X_{A_{1}}^{\alpha}-L_{\overline{A}_{0}}^{\alpha}\right), \label{eq:LoAba}\\
R_{\overline{A}_{\beta}}^{\alpha}&=Y_{A_{1}}^{\alpha}+(1-\beta)\left(R_{\overline{A}_{0}}^{\alpha}-Y_{A_{1}}^{\alpha}\right), \label{eq:RoAba}\\
L_{\overline{A}_{0}}^{\alpha}&=C_{{\cal A}}-(1-\alpha)\left(C_{{\cal A}}-L_{\overline{A}_{0}}\right), \label{eq:LoA0a}\\
R_{\overline{A}_{0}}^{\alpha}&=C_{{\cal A}}+(1-\alpha)\left(R_{\overline{A}_{0}}-C_{{\cal A}}\right) \label{eq:RoA0a}
\end{align}
where
\begin{align}
X_{A_1}^{\alpha}&=C_{{\cal A}}-(1-\alpha)(C_{{\cal A}}-X_{A_1}), \label{eq:XA1a}\\
Y_{A_1}^{\alpha}&=C_{{\cal A}}+(1-\alpha)(Y_{A_1}-C_{{\cal A}}) \label{eq:YA1a}
\end{align}
are, in this paper, called the {\it left principle number} and {\it right principle number} of ${\cal A}$ respectively and $C_{{\cal A}}$ stands for the core of ${\cal A}$, that is, the crisp value $[{\cal A}]_{1}^{1}$.
These meet
\begin{align*}
L_{\overline{A}_{0}}^{\alpha}
\le
X_{A_{1}}^{\alpha}
\le 
L_{\underline{A}_{0}}^{\alpha}
\le 
C_{{\cal A}}
\le
R_{\underline{A}_{0}}^{\alpha}
\le 
Y_{A_{1}}^{\alpha}
\le 
R_{\overline{A}_{0}}^{\alpha}.
\end{align*}
(See Figure \ref{fig:LXLCRYR} later.)
In particular, the supports of ${\cal A}$,
\begin{align*}
[\underline{A}_{0}]_{\alpha}=\left[L_{\underline{A}_{0}}^{\alpha},R_{\underline{A}_{0}}^{\alpha}\right]
\quad {\rm and}\quad 
[\overline{A}_{0}]_{\alpha}=\left[L_{\overline{A}_{0}}^{\alpha},R_{\overline{A}_{0}}^{\alpha}\right],
\end{align*}
represent the $\alpha$-cut sets of the lower and upper membership functions of ${\rm FP}({\cal A})$ respectively.
Also, 
\begin{align*}
[A_{1}]_{\alpha}
=
\left[X_{A_{1}}^{\alpha},Y_{A_{1}}^{\alpha}\right]
\end{align*}
is the $\alpha$-cut set of ${\rm P}({\cal A})$.
The triangular type-1 number $u$ is determined by its left end $l$, core $c$ and right end $r$:
\begin{align*}
u=\langle\!\langle l,c,r \rangle\!\rangle,
\end{align*}
but the triangular perfect quasi-type-2 fuzzy number ${\cal A}$ is determined by its upper left end $L_{\overline{A}_0}$, left principle number $X_{A_1}$, lower left end $L_{\underline{A}_0}$, core $C_{{\cal A}}$, lower right end $R_{\underline{A}_0}$, right principle number $Y_{A_1}$ and upper right end $R_{\overline{A}_0}$:
\begin{align*}
{\cal A}=\langle\!\langle L_{\overline{A}_0},X_{A_1}, L_{\underline{A}_0};C_{{\cal A}};R_{\underline{A}_0},Y_{A_1},R_{\overline{A}_0} \rangle\!\rangle.
\end{align*}

\begin{figure}[h]
\begin{center}
\includegraphics[width=7cm]{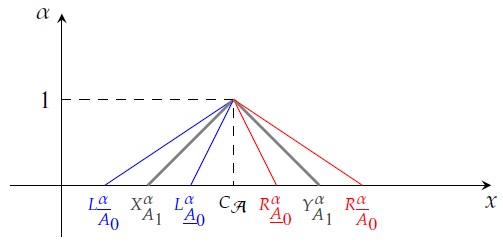}
\caption{A view of ${\cal A}$ from directly above}
\label{fig:LXLCRYR}
\end{center}
\end{figure}

We here reconfirm the significance of type-2 fuzzy numbers.
For example, let us denote 
\[
\underset{\sim}{3}=\mbox{`about $3$'}
\]
by a triangular fuzzy number.
Then, the core of $\underset{\sim}{3}$ is of course $3$, but how should we determine the left and right ends of it?
The simple representation of $\underset{\sim}{3}$ is the (more precise) isosceles triangular fuzzy number with some $\delta>0$:
\[
\underset{\sim}{3}=\langle\!\langle 3-\delta,3,3+\delta \rangle\!\rangle.
\]
So it is important to determine this $\delta$ appropriately, but it is generally difficult to determine $\delta$ objectively.
By setting $3-\delta$ (resp. $3+\delta$) as the left (resp. right) principle number and reconsidering `about $3$' as the triangular perfect quasi-type-2 fuzzy number, (\ref{eq:uAba})-(\ref{eq:RuAba}) and (\ref{eq:oAba})-(\ref{eq:RoAba}) thus determine subjective $\delta$.
Herein lies the necessity and usefulness of type-2 fuzzy notion.

\subsection{Type-2 Fuzzy Number-valued Functions}
We introduce the following Hung-Yang distance so as to consider the type-2 fuzzy topology.

\begin{df}[\cite{HY}]
Let ${\cal A},{\cal B}$ be type-2 fuzzy sets on $X$.
A distance between ${\cal A}$ and ${\cal B}$ is defined as
\begin{align}
d_{{\rm HY}}({\cal A},{\cal B})
:=\int_a^bH_{{\rm f}}(\kappa_{{\cal A}}(x),\kappa_{{\cal B}}(x))\,dx
\end{align}
where 
\begin{align*}
H_{{\rm f}}(\kappa_{{\cal A}}(x),\kappa_{{\cal B}}(x))&:=\frac{\int_0^1\beta d_{{\rm H}}(S_{{\cal A}}(x;\beta),S_{{\cal B}}(x;\beta))\,d\beta}{\int_0^1\beta \,d\beta} \\
&\hspace{1mm} =2\int_0^1\beta d_{{\rm H}}(S_{{\cal A}}(x;\beta),S_{{\cal B}}(x;\beta))\,d\beta
\end{align*}
and the above integrals are defined in the sense of Riemann.
We denote the space of type-2 fuzzy numbers on $X$ equipped with $d_{{\rm HY}}$-topology by $\mathscr{T}^2(X)$.
\end{df}

{\small
\begin{rem}
See Theorem 2.3 of \cite{MN} to make sure $\mathscr{T}^2(X)$ is a crisp metric space, that is, $d_{{\rm HY}}$ satisfies the metric axiom.
\end{rem}
}

\begin{df}[\cite{MN}, Definition 4.1]
Let ${\cal A},{\cal B}\in \mathscr{T}^2(X)$.
If there exists some ${\cal C}\in \mathscr{T}^2(X)$ such that 
\begin{align*}
{\cal A}={\cal B}+{\cal C},
\end{align*}
we call ${\cal C}$ the T2-Hukuhara difference of ${\cal A}$ and ${\cal B}$.
Then, we write ${\cal C}$ as ${\cal A}-{\cal B}$ as with type-1.
\end{df}

\begin{thm}[\cite{MN}, Theorem 4.1]
Let ${\cal A},{\cal B}\in \mathscr{T}^2(X)$.
We denote the $\beta$-cut sets of them by (\ref{eq:ABbetac}).
Then, the $\beta$-cut set of the T2-Hukuhara difference of ${\cal A}$ and ${\cal B}$ is the T1-Hukuhara difference of the upper and lower membership functions of ${\cal A},{\cal B}$:
\begin{align*}
[{\cal A}-{\cal B}]_{\beta}
=\left\langle \underline{(A-B)}_{\beta},\ \overline{(A-B)}_{\beta} \right\rangle
=\left\langle \underline{A}_{\beta}-\underline{B}_{\beta},\ \overline{A}_{\beta}-\overline{B}_{\beta} \right\rangle.
\end{align*}
\end{thm}

Zadeh's extension principle derives the type-2 fuzzy number-valued function ${\cal F}:\mathscr{T}^2(I)\to \mathscr{T}^2(\mathbb R)$ via a crisp function $f:I\to \mathbb R$ in the same way as type-1.
We consider, in this paper, the case of $\mathscr{T}^2(I)=I$.
Also, type-2 fuzzy number-valued functions are simply called type-2 fuzzy functions.
If the $\beta$-cut set of ${\cal F}:I\to \mathscr{T}^2(\mathbb R)$ is represented by
\begin{align*}
[{\cal F}(x)]_{\beta}:=\left\langle [\underline{F}_{\beta}(x)]_{\alpha},\ [\overline{F}_{\beta}(x)]_{\alpha} \right\rangle,
\end{align*}
we write 
\begin{align*}
[\underline{F}_{\beta}(x)]_{\alpha}&:=[\underline{F}_{\beta,-,\alpha}(x),\ \underline{F}_{\beta,+,\alpha}(x)], \\
[\overline{F}_{\beta}(x)]_{\alpha}&:=[\overline{F}_{\beta,-,\alpha}(x),\ \overline{F}_{\beta,+,\alpha}(x)]
\end{align*}
for all $x\in I$ and any $\alpha,\beta\in [0,1]$.

We use daggers $\dag,\ddag$ introduced in Appendix \ref{app:T1FNVF} as the symbol for type-2 fuzzy derivatives.

\begin{df}[\cite{MN}, Definition 4.4]
\label{df:T2Hd}
Let ${\cal F}:I\to \mathscr{T}^2(\mathbb R)$ and $h>0$ be a crisp number.
${\cal F}$ is T2-differentiable in the first form at some $x_0\in I$ if and only if there exist ${\cal F}(x_0+h)-{\cal F}(x_0)$ and ${\cal F}(x_0)-{\cal F}(x_0-h)$ satisfying that the fuzzy limit
\begin{align}
\label{eq:1T2d}
{\cal F}^\dag(x_0)
:=\lim_{h\downarrow 0}\frac{{\cal F}(x_0+h)-{\cal F}(x_0)}{h}
=\lim_{h\downarrow 0}\frac{{\cal F}(x_0)-{\cal F}(x_0-h)}{h}
\end{align}
exists.
Moreover, ${\cal F}$ is T2-differentiable in the second form at some $x_0\in I$ if and only if there exist ${\cal F}(x_0)-{\cal F}(x_0+h)$ and ${\cal F}(x_0-h)-{\cal F}(x_0)$ satisfying that the fuzzy limit
\begin{align}
\label{eq:2T2d}
{\cal F}^\ddag(x_0)
:=\lim_{h\uparrow 0}\frac{{\cal F}(x_0)-{\cal F}(x_0+h)}{-h}
=\lim_{h\uparrow 0}\frac{{\cal F}(x_0-h)-{\cal F}(x_0)}{-h}
\end{align}
exists.
Here the above differences (resp. limits) are due to the meaning of T2-Hukuhara (resp. $d_{{\rm HY}}$).
If ${\cal F}$ is T2-differentiable in both senses at any $x\in I$, ${\cal F}^\dag$ and ${\cal F}^\ddag$ is called the (1)-T2-derivative and (2)-T2-derivative of ${\cal F}$, respectively.
\end{df}

{\small
\begin{rem}
\quad
\begin{itemize}
\item[1)] Like Remark \ref{rem:thirdfourth}, we shall ignore T2-derivatives in the third and fourth forms.
The limits of both the forms become crisp numbers as with type-1.
This thing has already been mentioned in Note 4.1 of \cite{MN}.
\item[2)] As with type-1, second-order T2-derivatives are obtained by applying first-order T2-derivatives to (\ref{eq:1T2d}) and (\ref{eq:2T2d}).
\end{itemize}
\end{rem}
}

In what follows, ${\cal F}^{\dag\dag}$ is called the (1,1)-T2-derivative of ${\cal F}$, and the other cases are similar.

\begin{thm}[\cite{MN}, Theorem 4.2]
\label{thm:MNpara}
Let ${\cal F}:I\to \mathscr{T}^2(\mathbb R)$ be T2-differentiable on $I$.
Then, the parametric forms of its T2-derivatives are given by
\begin{itemize}
\item[1)] the (1)-parametric form:
\begin{align*}
[{\cal F}^\dag(x)]_{\beta}^{\alpha}&=\left\langle[\underline{F}^\dag_{\beta}(x)]_{\alpha},\ [\overline{F}^\dag_{\beta}(x)]_{\alpha}\right\rangle \\
&=\left\langle[\underline{F}'_{\beta,-,\alpha}(x),\underline{F}'_{\beta,+,\alpha}(x)],\ [\overline{F}'_{\beta,-,\alpha}(x),\overline{F}'_{\beta,+,\alpha}(x)]\right\rangle,
\end{align*}
\item[2)] the (2)-parametric form:
\begin{align*}
[{\cal F}^\ddag(x)]_{\beta}^{\alpha}&=\left\langle[\underline{F}^\ddag_{\beta}(x)]_{\alpha},\ [\overline{F}^\ddag_{\beta}(x)]_{\alpha}\right\rangle \\
&=\left\langle[\underline{F}'_{\beta,+,\alpha}(x),\underline{F}'_{\beta,-,\alpha}(x)],\ [\overline{F}'_{\beta,+,\alpha}(x),\overline{F}'_{\beta,-,\alpha}(x)]\right\rangle.
\end{align*}
\end{itemize}
\end{thm}

\begin{thm}[\cite{MN}, Corollary 4.1]
\label{thm:MNtpq}
Let ${\cal F}:I\to \mathscr{QT}^2(\mathbb R)$ be triangular, that is, 
\begin{align*}
{\cal F}(x)=\langle\!\langle L_{\overline{F}_0(x)},X_{F_1(x)},L_{\underline{F}_0(x)};C_{{\cal F}(x)};R_{\underline{F}_0(x)},Y_{F_1(x)},R_{\overline{F}_0(x)} \rangle\!\rangle.
\end{align*}
\begin{itemize}
\item[1)] If ${\cal F}$ is (1)-T2-differentiable on $I$, then
\begin{align*}
{\cal F}^{\dag}(x)=\langle\!\langle L_{\overline{F}_0^{\dag}(x)},X_{F_1^{\dag}(x)},L_{\underline{F}_0^{\dag}(x)};C_{{\cal F}^{\dag}(x)};R_{\underline{F}_0^{\dag}(x)},Y_{F_1^{\dag}(x)},R_{\overline{F}_0^{\dag}(x)} \rangle\!\rangle.
\end{align*}
\item[2)] If ${\cal F}$ is (2)-T2-differentiable on $I$, then
\begin{align*}
{\cal F}^{\ddag}(x)=\langle\!\langle R_{\overline{F}_0^{\ddag}(x)},Y_{F_1^{\ddag}(x)},R_{\underline{F}_0^{\ddag}(x)};C_{{\cal F}^{\ddag}(x)};L_{\underline{F}_0^{\ddag}(x)},X_{F_1^{\ddag}(x)},L_{\overline{F}_0^{\ddag}(x)} \rangle\!\rangle.
\end{align*}
\end{itemize}
\end{thm}

\section{Main Theorems and the Proofs}

\subsection{Second-order Differentiation and Continuity of Type-2 Fuzzy Number-valued Functions}
\begin{thm}
\label{thm:FT2sdba}
Let ${\cal F}:I\to \mathscr{T}^2(\mathbb R)$ be second-order T2-differentiable on $I$.
Then, the parametric forms of its second-order T2-derivatives are given by
\begin{itemize}
\item[1)] the (1,1)-parametric form:
\begin{align*}
[{\cal F}^{\dag\dag}(x)]_{\beta}^{\alpha}&=\left\langle[\underline{F}^{\dag\dag}_{\beta}(x)]_{\alpha},\ [\overline{F}^{\dag\dag}_{\beta}(x)]_{\alpha}\right\rangle \\
&=\left\langle[\underline{F}''_{\beta,-,\alpha}(x),\underline{F}''_{\beta,+,\alpha}(x)],\ [\overline{F}''_{\beta,-,\alpha}(x),\overline{F}''_{\beta,+,\alpha}(x)]\right\rangle,
\end{align*}
\item[2)] the (1,2)-parametric form:
\begin{align*}
[{\cal F}^{\dag\ddag}(x)]_{\beta}^{\alpha}&=\left\langle[\underline{F}^{\dag\ddag}_{\beta}(x)]_{\alpha},\ [\overline{F}^{\dag\ddag}_{\beta}(x)]_{\alpha}\right\rangle \\
&=\left\langle[\underline{F}''_{\beta,+,\alpha}(x),\underline{F}''_{\beta,-,\alpha}(x)],\ [\overline{F}''_{\beta,+,\alpha}(x),\overline{F}''_{\beta,-,\alpha}(x)]\right\rangle,
\end{align*}
\item[3)] the (2,1)-parametric form:
\begin{align*}
[{\cal F}^{\ddag\dag}(x)]_{\beta}^{\alpha}&=\left\langle[\underline{F}^{\ddag\dag}_{\beta}(x)]_{\alpha},\ [\overline{F}^{\ddag\dag}_{\beta}(x)]_{\alpha}\right\rangle \\
&=\left\langle[\underline{F}''_{\beta,+,\alpha}(x),\underline{F}''_{\beta,-,\alpha}(x)],\ [\overline{F}''_{\beta,+,\alpha}(x),\overline{F}''_{\beta,-,\alpha}(x)]\right\rangle,
\end{align*}
\item[4)] the (2,2)-parametric form:
\begin{align*}
[{\cal F}^{\ddag\ddag}(x)]_{\beta}^{\alpha}&=\left\langle[\underline{F}^{\ddag\ddag}_{\beta}(x)]_{\alpha},\ [\overline{F}^{\ddag\ddag}_{\beta}(x)]_{\alpha}\right\rangle \\
&=\left\langle[\underline{F}''_{\beta,-,\alpha}(x),\underline{F}''_{\beta,+,\alpha}(x)],\ [\overline{F}''_{\beta,-,\alpha}(x),\overline{F}''_{\beta,+,\alpha}(x)]\right\rangle.
\end{align*}
\end{itemize}
\end{thm}

\begin{proof}
The proof can be obtained by substituting ${\cal F}^{\dag}$ or ${\cal F}^{\ddag}$ for ${\cal F}$ in Theorem \ref{thm:MNpara}.
\end{proof}

\begin{thm}
Let ${\cal F}:I\to \mathscr{QT}^2(\mathbb R)$ be triangular, that is, 
\begin{align*}
{\cal F}(x)=\langle\!\langle L_{\overline{F}_0(x)},X_{F_1(x)},L_{\underline{F}_0(x)};C_{{\cal F}(x)};R_{\underline{F}_0(x)},Y_{F_1(x)},R_{\overline{F}_0(x)} \rangle\!\rangle.
\end{align*}
\begin{itemize}
\item[1)] If ${\cal F}$ is (1,1)-T2-differentiable on $I$, then
\begin{align*}
{\cal F}^{\dag\dag}(x)=\langle\!\langle L_{\overline{F}_0^{\dag\dag}(x)},X_{F_1^{\dag\dag}(x)},L_{\underline{F}_0^{\dag\dag}(x)};C_{{\cal F}^{\dag\dag}(x)};R_{\underline{F}_0^{\dag\dag}(x)},Y_{F_1^{\dag\dag}(x)},R_{\overline{F}_0^{\dag\dag}(x)} \rangle\!\rangle.
\end{align*}
\item[2)] If ${\cal F}$ is (1,2)-T2-differentiable on $I$, then
\begin{align*}
{\cal F}^{\dag\ddag}(x)=\langle\!\langle R_{\overline{F}_0^{\dag\ddag}(x)},Y_{F_1^{\dag\ddag}(x)},R_{\underline{F}_0^{\dag\ddag}(x)};C_{{\cal F}^{\dag\ddag}(x)};L_{\underline{F}_0^{\dag\ddag}(x)},X_{F_1^{\dag\ddag}(x)},L_{\overline{F}_0^{\dag\ddag}(x)} \rangle\!\rangle.
\end{align*}
\item[3)] If ${\cal F}$ is (2,1)-T2-differentiable on $I$, then
\begin{align*}
{\cal F}^{\ddag\dag}(x)=\langle\!\langle R_{\overline{F}_0^{\ddag\dag}(x)},Y_{F_1^{\ddag\dag}(x)},R_{\underline{F}_0^{\ddag\dag}(x)};C_{{\cal F}^{\ddag\dag}(x)};L_{\underline{F}_0^{\ddag\dag}(x)},X_{F_1^{\ddag\dag}(x)},L_{\overline{F}_0^{\ddag\dag}(x)} \rangle\!\rangle.
\end{align*}
\item[4)] If ${\cal F}$ is (2,2)-T2-differentiable on $I$, then
\begin{align*}
{\cal F}^{\ddag\ddag}(x)=\langle\!\langle L_{\overline{F}_0^{\ddag\ddag}(x)},X_{F_1^{\ddag\ddag}(x)},L_{\underline{F}_0^{\ddag\ddag}(x)};C_{{\cal F}^{\ddag\ddag}(x)};R_{\underline{F}_0^{\ddag\ddag}(x)},Y_{F_1^{\ddag\ddag}(x)},R_{\overline{F}_0^{\ddag\ddag}(x)} \rangle\!\rangle.
\end{align*}
\end{itemize}
\end{thm}

\begin{proof}
The proof can be obtained by substituting ${\cal F}^{\dag}$ or ${\cal F}^{\ddag}$ for ${\cal F}$ in Theorem \ref{thm:MNtpq}.
\end{proof}

\begin{df}
Let ${\cal F}:I\to \mathscr{T}^2(\mathbb R)$ and $h>0$ be a crisp number.
${\cal F}$ is T2-continuous on $I$ if and only if there exists the limit in $d_{{\rm HY}}$:
\begin{align}
\label{eq:conti}
\lim_{h\downarrow 0}\{{\cal F}(x+h)-{\cal F}(x)\}
=\lim_{h\downarrow 0}\{{\cal F}(x)-{\cal F}(x-h)\}
=0
\end{align}
for any $x\in I$.
We write ${\cal C}(I;\mathscr{T}^2(\mathbb R))$ for the space of T2-continuous fuzzy functions.
\end{df}

\begin{thm}
\label{thm:difconti}
If ${\cal F}:I\to \mathscr{T}^2(\mathbb R)$ is T2-differentiable on $I$, then ${\cal F}$ is T2-continuous on $I$.
\end{thm}

\begin{proof}
${\cal F}$ is T2-differentiable on $I$ by the assumption, so the limit
\begin{align*}
\lim_{h\downarrow 0}\frac{{\cal F}(x+h)-{\cal F}(x)}{h}=\lim_{h\downarrow 0}\frac{{\cal F}(x)-{\cal F}(x-h)}{h}
\end{align*}
exists.
Also, we can denote
\begin{align}
{\cal F}(x+h)-{\cal F}(x)&=\frac{{\cal F}(x+h)-{\cal F}(x)}{h}h \label{eq:diffFx+h-Fx}\\
{\cal F}(x)-{\cal F}(x-h)&=\frac{{\cal F}(x)-{\cal F}(x-h)}{h}h \label{eq:diffFx-Fx-h}
\end{align}
for any $x\in I$.
Hence, we have (\ref{eq:conti}) by letting $h\downarrow 0$ in both sides of (\ref{eq:diffFx+h-Fx}) and (\ref{eq:diffFx-Fx-h}).
\end{proof}

\subsection{Differentiation of Four Rules of Type-2 Fuzzy Number-valued Functions}
We say in this paper that ${\cal F}:I\to \mathscr{T}^2(\mathbb R)$ is second-order T2-differentiable on $I$ in the same case of differentiability, if ${\cal F}^{\dag}$ and ${\cal F}^{\ddag}$ are (1)-T2-differentiable and (2)-T2-differentiable on $I$, respectively.

\begin{thm}
\label{thm:F+GsHdd}
Let ${\cal F},{\cal G}:I\to \mathscr{T}^2(\mathbb R)$ be second-order T2-differentiable on $I$ in the same case of differentiability.
Then, ${\cal F}+{\cal G}:I\to \mathscr{T}^2(\mathbb R)$ is second-order T2-differentiable on $I$ and 
\begin{align}
({\cal F}+{\cal G})^{\dag}(x)&={\cal F}^{\dag}(x)+{\cal G}^{\dag}(x), \label{eq:F+Gd}\\
({\cal F}+{\cal G})^{\ddag}(x)&={\cal F}^{\ddag}(x)+{\cal G}^{\ddag}(x), \\
({\cal F}+{\cal G})^{\dag\dag}(x)&={\cal F}^{\dag\dag}(x)+{\cal G}^{\dag\dag}(x), \\
({\cal F}+{\cal G})^{\ddag\ddag}(x)&={\cal F}^{\ddag\ddag}(x)+{\cal G}^{\ddag\ddag}(x) \label{eq:F+Gdddd}
\end{align}
for $x\in I$.
Moreover, if there exists the T2-Hukuhara difference ${\cal F}-{\cal G}$, then all of (\ref{eq:F+Gd})-(\ref{eq:F+Gdddd}) hold even if $+$ is replaced with the T2-Hukuhara difference $-$.
\end{thm}

\begin{proof}
It is obvious for sums from Definition \ref{df:T2Hd}.
We prove only
\begin{align}
({\cal F}-{\cal G})^{\dag}(x)&={\cal F}^{\dag}(x)-{\cal G}^{\dag}(x), \label{eq:HF-Gd}\\
({\cal F}-{\cal G})^{\dag\dag}(x)&={\cal F}^{\dag\dag}(x)-{\cal G}^{\dag\dag}(x) \label{eq:HF-Gdd}
\end{align}
for $x\in I$.
The other two cases can be shown in the same way.
Since ${\cal F}-{\cal G}$ exists, there is some type-2 fuzzy function ${\cal W}:I\to \mathscr{T}^2(\mathbb R)$ such that
\begin{align*}
{\cal F}(x)&={\cal G}(x)+{\cal W}(x), \\
{\cal F}(x\pm h)&={\cal G}(x\pm h)+{\cal W}(x\pm h)
\end{align*}
where $h>0$ is a crisp number.
Then, we have
\begin{align}
\label{eq:Fx+h-Fx}
\begin{aligned}
{\cal F}(x+h)-{\cal F}(x)&=({\cal G}(x+h)+{\cal W}(x+h))-({\cal G}(x)+{\cal W}(x)) \\
&=({\cal G}(x+h)-{\cal G}(x))+({\cal W}(x+h)-{\cal W}(x))
\end{aligned}
\end{align}
by virtue of Lemma \ref{lem:Hddis} mentioned later.
Similarly, we also have
\begin{align}
\label{eq:Fx-Fx-h}
{\cal F}(x)-{\cal F}(x-h)
=({\cal G}(x)-{\cal G}(x-h))+({\cal W}(x)-{\cal W}(x-h)).
\end{align}
${\cal F}$ and ${\cal G}$ are T2-differentiable on $I$, so ${\cal F}(x+h)-{\cal F}(x)$, ${\cal F}(x)-{\cal F}(x-h)$, ${\cal G}(x+h)-{\cal G}(x)$ and ${\cal G}(x)-{\cal G}(x-h)$ exist for any $x\in I$.
Thus, the following derivative in $d_{{\rm HY}}$ exists:
\begin{align*}
({\cal F}-{\cal G})^{\dag}(x)=\lim_{h\downarrow 0}\frac{{\cal W}(x+h)-{\cal W}(x)}{h}=\lim_{h\downarrow 0}\frac{{\cal W}(x)-{\cal W}(x-h)}{h},\quad x\in I.
\end{align*}
This limit is ${\cal F}^{\dag}(x)-{\cal G}^{\dag}(x)$, $x\in I$, from (\ref{eq:Fx+h-Fx}) and (\ref{eq:Fx-Fx-h}).
Hence we have gained (\ref{eq:HF-Gd}).
(\ref{eq:HF-Gdd}) can be obtained by repeating the above discussion for $({\cal F}-{\cal G})^{\dag}$.
This completes the proof.
\end{proof}

\begin{lem}
\label{lem:Hddis}
Let $u_j,v_j\in \mathscr{T}^1(\mathbb R)\ (j=1,2)$.
Suppose that $(u_1+v_1)-(u_2+v_2)$, $u_1-u_2$ and $v_1-v_2$ exist.
Then, the distributive law for T1-Hukuhara differences in the following sense hold:
\begin{align*}
(u_1+v_1)-(u_2+v_2)=(u_1-u_2)+(v_1-v_2).
\end{align*}
Moreover, the same result holds for type-2 fuzzy numbers.
\end{lem}

\begin{proof}
We prove only for type-1 fuzzy numbers, since the definition of T2-Hukuhara differences for type-2 fuzzy numbers is essentially equal to that for type-1 fuzzy numbers.
The T1-Hukuhara difference is the difference between the left ends and the right ends of two intervals.
Let the $\alpha$-cuts of $u_j,v_j\ (j=1,2)$ be represented as
\begin{align*}
[u_j]_{\alpha}:=[u_{j,-}(\alpha),u_{j,+}(\alpha)],\quad 
[v_j]_{\alpha}:=[v_{j,-}(\alpha),v_{j,+}(\alpha)].
\end{align*}
Then, we have the following and finish the proof:
\begin{align*}
&[(u_1+v_1)-(u_2+v_2)]_{\alpha} \\
&=[u_1+v_1]_{\alpha}-[u_2+v_2]_{\alpha} \\
&=([u_1]_{\alpha}+[v_1]_{\alpha})-([u_2]_{\alpha}+[v_2]_{\alpha}) \\
&=([u_{1,-}(\alpha),u_{1,+}(\alpha)]+[v_{1,-}(\alpha),v_{1,+}(\alpha)])-([u_{2,-}(\alpha),u_{2,+}(\alpha)]+[v_{2,-}(\alpha),v_{2,+}(\alpha)]) \\
&=[u_{1,-}(\alpha)+v_{1,-}(\alpha),\ u_{1,+}(\alpha)+v_{1,+}(\alpha)]-[u_{2,-}(\alpha)+v_{2,-}(\alpha),\ u_{2,+}(\alpha)+v_{2,+}(\alpha)] \\
&=[(u_{1,-}(\alpha)-u_{2,-}(\alpha))+(v_{1,-}(\alpha)-v_{2,-}(\alpha)),\ (u_{1,+}(\alpha)-u_{2,+}(\alpha))+(v_{1,+}(\alpha)-v_{2,+}(\alpha))] \\
&=[u_{1,-}(\alpha)-u_{2,-}(\alpha),\ u_{1,+}(\alpha)-u_{2,+}(\alpha)]+[v_{1,-}(\alpha)-v_{2,-}(\alpha),\ v_{1,+}(\alpha)-v_{2,+}(\alpha)] \\
&=[u_1-u_2]_{\alpha}+[v_1-v_2]_{\alpha} \\
&=[(u_1-u_2)+(v_1-v_2)]_{\alpha}.
\end{align*}
\end{proof}

{\small
\begin{rem}
We can generally have the similar results to Theorem \ref{thm:F+GsHdd} for any order $N\in \mathbb N\cup \{0\}$.
\end{rem}
}

\begin{thm}
\label{thm:F-GT2d}
Let ${\cal F},{\cal G}:I\to \mathscr{T}^2(\mathbb R)$ be second-order T2-differentiable on $I$ such that
\begin{itemize}
\item if ${\cal F}$ is (1,1)-T2-differentiable then ${\cal G}$ is (2,1)-T2-differentiable, or
\item if ${\cal F}$ is (1,2)-T2-differentiable then ${\cal G}$ is (2,2)-T2-differentiable, or
\item if ${\cal F}$ is (2,1)-T2-differentiable then ${\cal G}$ is (1,1)-T2-differentiable, or
\item if ${\cal F}$ is (2,2)-T2-differentiable then ${\cal G}$ is (1,2)-T2-differentiable.
\end{itemize}
Suppose that there exists the T2-Hukuhara difference ${\cal F}(x)-{\cal G}(x)$ for any $x\in I$.
Then, ${\cal F}-{\cal G}:I\to \mathscr{T}^2(\mathbb R)$ is second-order T2-differentiable on $I$ and 
\begin{align}
({\cal F}-{\cal G})^{\dag}(x)&={\cal F}^{\dag}(x)+(-1){\cal G}^{\ddag}(x), \label{eq:F-Gd}\\
({\cal F}-{\cal G})^{\ddag}(x)&={\cal F}^{\ddag}(x)+(-1){\cal G}^{\dag}(x), \\
({\cal F}-{\cal G})^{\dag\dag}(x)&={\cal F}^{\dag\dag}(x)+(-1){\cal G}^{\ddag\dag}(x), \label{eq:F-Gdd}\\
({\cal F}-{\cal G})^{\dag\ddag}(x)&={\cal F}^{\dag\ddag}(x)+(-1){\cal G}^{\ddag\ddag}(x), \\
({\cal F}-{\cal G})^{\ddag\dag}(x)&={\cal F}^{\ddag\dag}(x)+(-1){\cal G}^{\dag\dag}(x), \\
({\cal F}-{\cal G})^{\ddag\ddag}(x)&={\cal F}^{\ddag\ddag}(x)+(-1){\cal G}^{\dag\ddag}(x)
\end{align}
for $x\in I$.
\end{thm}

\begin{proof}
We prove only (\ref{eq:F-Gd}) and (\ref{eq:F-Gdd}) in the first case. 
The other cases and equations can be similarly proved.

Let us first prove (\ref{eq:F-Gd}) in the first case.
Since, on $I$, ${\cal F}$ is (1)-T2-differentiable and ${\cal G}$ is (2)-T2-differentiable, ${\cal F}(x+h)-{\cal F}(x)$, ${\cal F}(x)-{\cal F}(x-h)$, ${\cal G}(x)-{\cal G}(x+h)$ and ${\cal G}(x-h)-{\cal G}(x)$ exist for any $x\in I$.
Moreover, there exists a type-2 fuzzy number-valued function ${\cal W}_1$ such that
\begin{align*}
{\cal F}(x+h)&={\cal F}(x)+{\cal W}_1(x,h), \\
{\cal F}(x)&={\cal F}(x-h)+{\cal W}_2(x,h),
\end{align*}
for $h>0$, with
\begin{align*}
\lim_{h\downarrow 0}\frac{{\cal W}_1(x,h)}{h}
=\lim_{h\downarrow 0}\frac{{\cal W}_2(x,h)}{h}
={\cal F}^{\dag}(x)
\end{align*}
for $x\in I$ and
\begin{align*}
{\cal G}(x)&={\cal G}(x+h)+{\cal W}_3(x,h), \\
{\cal G}(x-h)&={\cal G}(x)+{\cal W}_4(x,h),
\end{align*}
for $h>0$, with 
\begin{align*}
\lim_{h\downarrow 0}\frac{{\cal W}_3(x,h)}{h}
=\lim_{h\downarrow 0}\frac{{\cal W}_4(x,h)}{h}
=(-1){\cal G}^{\ddag}(x)
\end{align*}
for $x\in I$.
Thus, we have
\begin{align*}
{\cal F}(x+h)+{\cal G}(x)={\cal F}(x)+{\cal G}(x+h)+{\cal W}_1(x,h)+{\cal W}_3(x,h),
\end{align*}
but we obtain
\begin{align*}
{\cal F}(x+h)-{\cal G}(x+h)={\cal F}(x)-{\cal G}(x)+{\cal W}_1(x,h)+{\cal W}_3(x,h)
\end{align*}
since ${\cal F}(x)-{\cal G}(x)$ exists for any $x\in I$ by the assumption.
This implies that $\{{\cal F}(x+h)-{\cal G}(x+h)\}-\{{\cal F}(x)-{\cal G}(x)\}$ exists and
\begin{align}
\label{eq:fracFh-Gh}
\frac{\{{\cal F}(x+h)-{\cal G}(x+h)\}-\{{\cal F}(x)-{\cal G}(x)\}}{h}=\frac{{\cal W}_1(x,h)}{h}+\frac{{\cal W}_3(x,h)}{h}
\end{align}
for any $x\in I$.
Similarly, we can find the existence of $\{{\cal F}(x)-{\cal G}(x)\}-\{{\cal F}(x-h)-{\cal G}(x-h)\}$ and obtain
\begin{align}
\label{eq:fracF-G}
\frac{\{{\cal F}(x)-{\cal G}(x)\}-\{{\cal F}(x-h)-{\cal G}(x-h)\}}{h}=\frac{{\cal W}_2(x,h)}{h}+\frac{{\cal W}_4(x,h)}{h}
\end{align}
for any $x\in I$.
Both the left-hand sides of (\ref{eq:fracFh-Gh}) and (\ref{eq:fracF-G}) have the common limit $({\cal F}-{\cal G})^{\dag}$ as $h\downarrow 0$.
Similarly, both the right-hand sides of (\ref{eq:fracFh-Gh}) and (\ref{eq:fracF-G}) have the common limit ${\cal F}^{\dag}+(-1){\cal G}^{\ddag}$ as $h\downarrow 0$.
Hence, we have obtained (\ref{eq:F-Gd}).

Let us next prove (\ref{eq:F-Gdd}) in the first case.
We can however derive it from (\ref{eq:F-Gd}) and Theorem \ref{thm:F+GsHdd} immediately.

This completes the proof.
\end{proof}

\begin{thm}
\label{thm:cftimesff}
Let $f:I\to \mathbb R$ be second-order differentiable and ${\cal G}:I\to \mathscr{T}^2(\mathbb R)$ be second-order T2-differentiable on $I$.
Then, 
\begin{itemize}
\item[1)] if $f(x)f'(x)>0$ and ${\cal G}$ is (1)-T2-differentiable, then $f{\cal G}$ is (1)-T2-differentiable and
\begin{align}
(f{\cal G})^{\dag}(x)
=f'(x){\cal G}(x)+f(x){\cal G}^{\dag}(x)
\end{align}
for $x\in I$;
\item[2)] if $f(x)f'(x)<0$ and ${\cal G}$ is (2)-T2-differentiable, then $f{\cal G}$ is (2)-T2-differentiable and
\begin{align}
\label{eq:f.Gdd}
(f{\cal G})^{\ddag}(x)
=f'(x){\cal G}(x)+f(x){\cal G}^{\ddag}(x)
\end{align}
for $x\in I$;
\item[3)] if $f(x)f'(x)>0$, $f'(x)f''(x)>0$ and ${\cal G}$ is (1,1)-T2-differentiable, then $f{\cal G}$ is (1,1)-T2-differentiable and
\begin{align}
(f{\cal G})^{\dag\dag}(x)
=f''(x){\cal G}(x)+2f'(x){\cal G}^{\dag}(x)+f(x){\cal G}^{\dag\dag}(x)
\end{align}
for $x\in I$;
\item[4)] if $f(x)f'(x)<0$, $f'(x)f''(x)<0$ and ${\cal G}$ is (2,2)-T2-differentiable, then $f{\cal G}$ is (2,2)-T2-differentiable and
\begin{align}
\label{eq:f.Gdddd}
(f{\cal G})^{\ddag\ddag}(x)
=f''(x){\cal G}(x)+2f'(x){\cal G}^{\ddag}(x)+f(x){\cal G}^{\ddag\ddag}(x)
\end{align}
for $x\in I$.
\end{itemize}
\end{thm}

\begin{proof}
We prove only 2) and 4) because the other two cases are proved in the same way.
\begin{itemize}
\item[2)] We suppose that $f(x)>0$ and $f'(x)<0$.
Let $h>0$ be a crisp number.
Since $f$ is differentiable, there exist $\varepsilon_j(x,h)\ (j=1,2)$ such that
\begin{align}
f(x)&=f(x+h)+\varepsilon_1(x,h), \label{eq:1approx}\\
f(x-h)&=f(x)+\varepsilon_2(x,h). \label{eq:2approx}
\end{align}
(Recall the first-order approximation of the differentiable function.)
Remark that $\varepsilon_1(x,h)>0$ since $\varepsilon_1(x,h)=f(x)-f(x+h)$ and $f$ is monotone decreasing.
The same applies to $\varepsilon_2(x,h)$.
Since ${\cal G}$ is (2)-T2-differentiable on $I$, ${\cal G}(x)-{\cal G}(x+h)$ and ${\cal G}(x-h)-{\cal G}(x)$ exist for any $x\in I$, that is, there exist type-2 fuzzy functions ${\cal W}_j\ (j=1,2)$ such that
\begin{align*}
{\cal G}(x)&={\cal G}(x+h)+{\cal W}_1(x,h), \\
{\cal G}(x-h)&={\cal G}(x)+{\cal W}_2(x,h).
\end{align*}
We thus have
\begin{align}
{\cal G}(x)&={\cal G}(x+h)+{\cal W}_1(x,h), \label{eq:GGx+h1}\\
{\cal G}(x-h)&={\cal G}(x)+{\cal W}_2(x,h). \label{eq:GGx+h2}
\end{align}
One has
\begin{align*}
f(x){\cal G}(x)&=\{f(x+h)+\varepsilon_1(x,h)\}\{{\cal G}(x+h)+{\cal W}_1(x,h)\} \\
&=f(x+h){\cal G}(x+h)+f(x+h){\cal W}_1(x,h)+\varepsilon_1(x,h){\cal G}(x+h)+\varepsilon_1(x,h){\cal W}_1(x,h)
\end{align*}
by virtue of (\ref{eq:1approx}) and (\ref{eq:GGx+h1}), so $f(x){\cal G}(x)-f(x+h){\cal G}(x+h)$ exists and
\begin{align}
\label{eq:fGx-fGx+h}
\begin{aligned}
&\frac{f(x){\cal G}(x)-f(x+h){\cal G}(x+h)}{-h} \\
&\qquad =f(x+h)\frac{{\cal W}_1(x,h)}{-h}+\frac{\varepsilon_1(x,h)}{-h}{\cal G}(x+h)+\frac{\varepsilon_1(x,h)}{-h}{\cal W}_1(x,h)
\end{aligned}
\end{align}
for any $x\in I$.
Similarly, from (\ref{eq:2approx}) and (\ref{eq:GGx+h2}), $f(x-h){\cal G}(x-h)-f(x){\cal G}(x)$ exists and we also have
\begin{align}
\label{eq:fGx-h-fGx}
\begin{aligned}
&\frac{f(x-h){\cal G}(x-h)-f(x){\cal G}(x)}{-h} \\
&\qquad =f(x-h)\frac{{\cal W}_2(x,h)}{-h}+\frac{\varepsilon_2(x,h)}{-h}{\cal G}(x-h)+\frac{\varepsilon_2(x,h)}{-h}{\cal W}_2(x,h)
\end{aligned}
\end{align}
for any $x\in I$.
Theorem \ref{thm:difconti} implies ${\cal G}\in {\cal C}(I;\mathscr{T}^2(\mathbb R))$, so ${\cal W}_j(x,h)\to 0\ (j=1,2)$ as $h\downarrow 0$.
Therefore (\ref{eq:f.Gdd}) is obtained by letting $h\downarrow 0$ in both sides of (\ref{eq:fGx-fGx+h}) and (\ref{eq:fGx-h-fGx}).
We can also gain the same result in the case that $f(x)<0$ and $f'(x)>0$.
Hence, (\ref{eq:f.Gdd}) has been derived.
\item[4)] (\ref{eq:f.Gdddd}) is immediately proved by repeating (\ref{eq:f.Gdd}).
\end{itemize}
This completes the proof.
\end{proof}

\subsection{Differentiation of Composite Type-2 Fuzzy Number-valued Functions}
We finally deal with the composition of crisp and type-1 / type-2 fuzzy functions and its derivative.

\begin{thm}
\label{thm:compo}
Let $f:I\to \mathbb R$ be differentiable and $G:\mathbb R\to \mathscr{T}^1(\mathbb R)$ T1-differentiable.
We consider the type-1 fuzzy composite function $G\circ f:I\to \mathscr{T}^1(\mathbb R)$.
Suppose that T1-Hukuhara differences 
\begin{align*}
(G\circ f)(x+h)-(G\circ f)(x)\quad {\rm and}\quad (G\circ f)(x)-(G\circ f)(x-h)
\end{align*}
exist for any $x\in I$ and $h>0$ sufficiently small.
Then, $G\circ f$ is T1-differentiable on $I$ and
\begin{align}
(G\circ f)^{\dag}&=G^{\dag}(f(x))f'(x), \label{eq:compo1}\\
(G\circ f)^{\ddag}&=G^{\ddag}(f(x))f'(x). \label{eq:compo2}
\end{align}
Moreover, the same result holds for ${\cal G}:\mathbb R\to \mathscr{T}^2(\mathbb R)$.
\end{thm}

\begin{proof}
We prove only for type-1 $G$ because we can prove for type-2 ${\cal G}$ in the same way.
Let $h>0$ be a crisp number.
The $\alpha$-cut set of the right difference quotient of $(G\circ f)$ is 
\begin{align*}
&\left[\frac{(G\circ f)(x+h)-(G\circ f)(x)}{h}\right]_{\alpha} \\
&=\left[\frac{(G\circ f)_{-,\alpha}(x+h)-(G\circ f)_{-,\alpha}(x)}{h},\ \frac{(G\circ f)_{+,\alpha}(x+h)-(G\circ f)_{+,\alpha}(x)}{h}\right] \\
&=\left[\frac{G_{-,\alpha}(f(x+h))-G_{-,\alpha}(f(x))}{h},\ \frac{G_{+,\alpha}(f(x+h))-G_{+,\alpha}(f(x))}{h}\right] \\
&=\left[\frac{G_{-,\alpha}(f(x+h))-G_{-,\alpha}(f(x))}{f(x+h)-f(x)}\frac{f(x+h)-f(x)}{h},\ \frac{G_{+,\alpha}(f(x+h))-G_{+,\alpha}(f(x))}{f(x+h)-f(x)}\frac{f(x+h)-f(x)}{h}\right].
\end{align*}
Similarly, the $\alpha$-cut set of the left difference quotient of $(G\circ f)$ is 
\begin{align*}
&\left[\frac{(G\circ f)(x)-(G\circ f)(x-h)}{h}\right]_{\alpha} \\
&=\left[\frac{G_{-,\alpha}(f(x))-G_{-,\alpha}(f(x-h))}{f(x)-f(x-h)}\frac{f(x)-f(x-h)}{h},\ \frac{G_{+,\alpha}(f(x))-G_{+,\alpha}(f(x-h))}{f(x)-f(x-h)}\frac{f(x)-f(x-h)}{h}\right].
\end{align*}
The rightmost-hand sides of the above two equations converge as $h\downarrow 0$, since $f(x\pm h)\to f(x)$ as $h\downarrow 0$ by virtue of the continuity of $f$.
Thus (\ref{eq:compo1}) can be obtained. 
(\ref{eq:compo2}) is similar.
Hence, this completes the proof.
\end{proof}

\section{Type-2 Fuzzy Initial Value Problems for Second-order T2FDEs \label{sec:ExT2FDE}}
We actually solve, in this section, some concrete type-2 fuzzy initial value problems for second-order T2FDEs.
We write all first-order Hukuhara derivatives $^{\dag}$ and $^{\ddag}$ together as ${\rm D}$.
That is, ${\rm D}^2$ denotes all second-order Hukuhara derivatives $^{\dag\dag}$, $^{\dag\ddag}$, $^{\ddag\dag}$ and $^{\ddag\ddag}$ together.
Italic numbers
\begin{align*}
{\it 0},{\it 1},{\it 2},{\it 3},{\it 4},{\it 5},{\it 6},{\it 7},{\it 8},{\it 9}
\end{align*}
stand for concrete type-2 fuzzy numbers.
We abbreviate the type-2 fuzzy initial value problem (resp. condition) as T2FIVP (resp. T2FIVC) in this subsection.

We consider T2FIVPs on $I=[0,r]$ for some $r>0$ or $I=[0,+\infty)$:
\begin{align}
\label{eq:gT2FDE1}
\begin{cases}
{\rm D}^2{\cal Y}(x)+a{\rm D}{\cal Y}(x)+b{\cal Y}(x)=0, \\
{\cal Y}(0)={\cal U}\in \mathscr{T}^2(\mathbb R), \\
{\rm D}{\cal Y}(0)={\cal V}\in \mathscr{T}^2(\mathbb R).
\end{cases}
\end{align}

\begin{df}
Let ${\cal Y}:I\to \mathscr{T}^2(\mathbb R)$ be second-order T2-differentiable.
We denote the $\beta$-cut set of ${\cal Y}={\cal Y}(x)$ by
\begin{align*}
[{\cal Y}(x)]_{\beta}=\left\langle \underline{Y}_{\beta}(x),\ \overline{Y}_{\beta}(x)\right\rangle.
\end{align*}
Then, ${\cal Y}$ is the $(i,j)$-type-2 fuzzy solution of (\ref{eq:gT2FDE1}), $(i,j)\in \{1,2\}^2$, if and only if, for each $\beta\in [0,1]$, 
\begin{itemize}
\item[i)] ${\rm D}_{i}\underline{Y}_{\beta}$, ${\rm D}_{i}\overline{Y}_{\beta}$, ${\rm D}^2_{i,j}\underline{Y}_{\beta}$, ${\rm D}^2_{i,j}\overline{Y}_{\beta}$ exist on $I$, and
\item[ii)] $\underline{Y}_{\beta}$ and $\overline{Y}_{\beta}:I\to \mathscr{T}^1(\mathbb R)$ satisfy
\begin{align*}
\begin{cases}
[{\rm D}^2_{i,j}\underline{Y}_{\beta}(x)]_{\alpha}+[a{\rm D}_{i}\underline{Y}_{\beta}(x)]_{\alpha}+[b\underline{Y}_{\beta}(x)]_{\alpha}=0,\quad x\in I, \\
\underline{Y}_{\beta}(0)=\underline{u}_{\beta}\in \mathscr{T}^1(\mathbb R), \\
{\rm D}_{i}\underline{Y}_{\beta}(0)=\underline{v}_{\beta}\in \mathscr{T}^1(\mathbb R)
\end{cases}
\end{align*}
and
\begin{align*}
\begin{cases}
[{\rm D}^2_{i,j}\overline{Y}_{\beta}(x)]_{\alpha}+[a{\rm D}_{i}\overline{Y}_{\beta}(x)]_{\alpha}+[b\overline{Y}_{\beta}(x)]_{\alpha}=0,\quad x\in I, \\
\overline{Y}_{\beta}(0)=\overline{u}_{\beta}\in \mathscr{T}^1(\mathbb R), \\
{\rm D}_{i}\overline{Y}_{\beta}(0)=\overline{v}_{\beta}\in \mathscr{T}^1(\mathbb R)
\end{cases}
\end{align*}
for any $\alpha\in [0,1]$, respectively.
\end{itemize}
\end{df}

{\small
\begin{rem}
The type-2 fuzzy solution generally becomes the type-1 fuzzy solution if $\beta=1$; it becomes the crisp solution if $\alpha=\beta=1$.
\end{rem}
}

We solve problems in the case of crisp coefficients.
We have four candidate solutions for type-1 fuzzy differential equations of order 2, but we have eight of them for type-2 ones of order 2.

Recall Definition \ref{df:+kop1} and \ref{df:+kop} for operations.
`({\tt A.C.})' represents `ANSWER COMPLETED'.

\subsection{Case of Positive Coefficients}
We begin with easy problems, that is, T2FIVPs of order 2 in the case of crisp coefficients.

\begin{prob}
\label{prob:P1}
Let ${\cal Y}:[0,1]\to \mathscr{T}^2(\mathbb R)$ be a type-2 fuzzy function.
Then, solve T2FIVPs:
\begin{numcases}
{}
{\rm D}^2{\cal Y}(x)+3{\rm D}{\cal Y}(x)=0, \label{eq:T2FDE1}\\
{\cal Y}(0)={\it 5}\in \mathscr{QT}^2(\mathbb R), \\
{\rm D}{\cal Y}(0)={\it 1}\in \mathscr{QT}^2(\mathbb R).
\end{numcases}
\end{prob}

\begin{sol}
We consider the $(\alpha,\beta)$-cut set of (\ref{eq:T2FDE1}):
\begin{align*}
\left\langle[\underline{Y}^{\dag\dag}_{\beta}(x)]_{\alpha}+3[\underline{Y}^{\dag}_{\beta}(x)]_{\alpha},\ [\overline{Y}^{\dag\dag}_{\beta}(x)]_{\alpha}+3[\overline{Y}^{\dag}_{\beta}(x)]_{\alpha}\right\rangle=0,
\quad \alpha\in [0,1].
\end{align*}
This can be solved by (1,1) or (2,2)-T1-differentiation, so we have the following two T1FIVPs:
\begin{numcases}
{}
[\underline{Y}^{\dag\dag}_{\beta}(x)]_{\alpha}+3[\underline{Y}^{\dag}_{\beta}(x)]_{\alpha}=[\underline{Y}''_{\beta,-,\alpha}(x)+3\underline{Y}'_{\beta,-,\alpha}(x),\ \underline{Y}''_{\beta,+,\alpha}(x)+3\underline{Y}'_{\beta,+,\alpha}(x)]=0, \label{eq:udd3d}\\
\underline{Y}_{\beta}(0)=\underline{5}_{\beta}\in \mathscr{T}^1(\mathbb R), \\
\underline{Y}_{\beta}^{\dag}(0)=\underline{1}_{\beta}\in \mathscr{T}^1(\mathbb R),
\end{numcases}
and
\begin{numcases}
{}
[\overline{Y}^{\dag\dag}_{\beta}(x)]_{\alpha}+3[\overline{Y}^{\dag}_{\beta}(x)]_{\alpha}=[\overline{Y}''_{\beta,-,\alpha}(x)+3\overline{Y}'_{\beta,-,\alpha}(x),\ \overline{Y}''_{\beta,+,\alpha}(x)+3\overline{Y}'_{\beta,+,\alpha}(x)]=0, \label{eq:odd3d}\\
\overline{Y}_{\beta}(0)=\overline{5}_{\beta}\in \mathscr{T}^1(\mathbb R), \\
\overline{Y}_{\beta}^{\dag}(0)=\overline{1}_{\beta}\in \mathscr{T}^1(\mathbb R).
\end{numcases}
We can thus solve the given T2FIVP because the solution method of type-1 fuzzy differential equations is well known.

(\ref{eq:udd3d}) can be solved in the form of
\begin{align}
\label{eq:uYbpma}
\underline{Y}_{\beta,\pm,\alpha}(x)=\underline{C}_{1,\pm}e^{-3x}+\underline{C}_{2,\pm},
\end{align}
thus,
\begin{align}
\label{eq:uYdbpma}
\underline{Y}_{\beta,\pm,\alpha}'(x)=-3\underline{C}_{1,\pm}e^{-3x},
\end{align}
where these equations represent two equations, one for the upper sign and the other for the lower sign.
Similarly, (\ref{eq:odd3d}) can be solved in the form of
\begin{align}
\label{eq:oYbpma}
\overline{Y}_{\beta,\pm,\alpha}(x)=\overline{C}_{1,\pm}e^{-3x}+\overline{C}_{2,\pm},
\end{align}
thus,
\begin{align}
\label{eq:oYdbpma}
\overline{Y}_{\beta,\pm,\alpha}'(x)=-3\overline{C}_{1,\pm}e^{-3x},
\end{align}
where these equations represent two equations, one for the upper sign and the other for the lower sign.

Let us determine type-2 fuzzy initial value ${\it 5}$ by setting them as the triangular quasi-type-2 fuzzy number
\begin{align*}
{\it 5}=\langle\!\langle 3.5,\ 4,\ 4.5\ ;\ 5\ ;\ 5.5,\ 6,\ 6.5\rangle\!\rangle.
\end{align*}
Since (\ref{eq:LoA0a}), (\ref{eq:LuA0a}), (\ref{eq:RuA0a}), (\ref{eq:RoA0a}), (\ref{eq:XA1a}) and (\ref{eq:YA1a}) imply that
\begin{align*}
L_{\overline{5}_{0}}^{\alpha}&=5-(1-\alpha)(5-3.5)=\frac{3}{2}\alpha+\frac{7}{2}, \\
L_{\underline{5}_{0}}^{\alpha}&=5-(1-\alpha)(5-4.5)=\frac{1}{2}\alpha+\frac{9}{2}, \\
R_{\underline{5}_{0}}^{\alpha}&=5+(1-\alpha)(5.5-5)=-\frac{1}{2}\alpha+\frac{11}{2}, \\
R_{\overline{5}_{0}}^{\alpha}&=5+(1-\alpha)(6.5-5)=-\frac{3}{2}\alpha+\frac{13}{2}, \\
X_{5_1}^{\alpha}&=5-(1-\alpha)(5-4)=\alpha+4, \\
Y_{5_1}^{\alpha}&=5+(1-\alpha)(6-5)=-\alpha+6,
\end{align*}
we have
\begin{align*}
L_{\underline{5}_{\beta}}^{\alpha}&=(\alpha+4)-(1-\beta)\left\{(\alpha+4)-\left(\frac{1}{2}\alpha+\frac{9}{2}\right)\right\}=\frac{1}{2}\alpha+\frac{1}{2}\alpha\beta-\frac{1}{2}\beta+\frac{9}{2}, \\
R_{\underline{5}_{\beta}}^{\alpha}&=(-\alpha+6)+(1-\beta)\left\{\left(-\frac{1}{2}\alpha+\frac{11}{2}\right)-(-\alpha+6)\right\}=-\frac{1}{2}\alpha-\frac{1}{2}\alpha\beta+\frac{1}{2}\beta+\frac{11}{2}
\end{align*}
and
\begin{align*}
L_{\overline{5}_{\beta}}^{\alpha}&=(\alpha+4)-(1-\beta)\left\{(\alpha+4)-\left(\frac{3}{2}\alpha+\frac{7}{2}\right)\right\}=\frac{3}{2}\alpha-\frac{1}{2}\alpha\beta+\frac{1}{2}\beta+\frac{7}{2}, \\
R_{\overline{5}_{\beta}}^{\alpha}&=(-\alpha+6)+(1-\beta)\left\{\left(-\frac{3}{2}\alpha+\frac{13}{2}\right)-(-\alpha+6)\right\}=-\frac{3}{2}\alpha+\frac{1}{2}\alpha\beta-\frac{1}{2}\beta+\frac{13}{2}
\end{align*}
from (\ref{eq:LuAba})-(\ref{eq:RuAba}) and (\ref{eq:LoAba})-(\ref{eq:RoAba}).
Therefore, it follows that
\begin{align}
[\underline{Y}_{\beta}(0)]_{\alpha}=[\underline{5}_{\beta}]_{\alpha}&=\left[\frac{1}{2}\alpha+\frac{1}{2}\alpha\beta-\frac{1}{2}\beta+\frac{9}{2},\ -\frac{1}{2}\alpha-\frac{1}{2}\alpha\beta+\frac{1}{2}\beta+\frac{11}{2}\right], \label{eq:u5ba}\\
[\overline{Y}_{\beta}(0)]_{\alpha}=[\overline{5}_{\beta}]_{\alpha}&=\left[\frac{3}{2}\alpha-\frac{1}{2}\alpha\beta+\frac{1}{2}\beta+\frac{7}{2},\ -\frac{3}{2}\alpha+\frac{1}{2}\alpha\beta-\frac{1}{2}\beta+\frac{13}{2}\right] \label{eq:o5ba}
\end{align}
from (\ref{eq:uAba}) and (\ref{eq:oAba}).
Next, let us determine type-2 fuzzy initial value ${\it 1}$ by setting them as the triangular quasi-type-2 fuzzy number
\begin{align*}
{\it 1}=\langle\!\langle -0.5,\ 0,\ 0.5\ ;\ 1\ ;\ 1.5,\ 2,\ 2.5\rangle\!\rangle.
\end{align*}
Therefore it follows that, in the same way as above, 
\begin{align}
[\underline{Y}_{\beta}^{\dag}(0)]_{\alpha}=[\underline{1}_{\beta}]_{\alpha}&=\left[\frac{1}{2}\alpha+\frac{1}{2}\alpha\beta-\frac{1}{2}\beta+\frac{1}{2},\ -\frac{1}{2}\alpha-\frac{1}{2}\alpha\beta+\frac{1}{2}\beta+\frac{3}{2}\right], \label{eq:u1ba}\\
[\overline{Y}_{\beta}^{\dag}(0)]_{\alpha}=[\overline{1}_{\beta}]_{\alpha}&=\left[\frac{3}{2}\alpha-\frac{1}{2}\alpha\beta+\frac{1}{2}\beta-\frac{1}{2},\ -\frac{3}{2}\alpha+\frac{1}{2}\alpha\beta-\frac{1}{2}\beta+\frac{5}{2}\right]. \label{eq:o1ba}
\end{align}

Substituting (\ref{eq:u5ba}) and (\ref{eq:u1ba}) for (\ref{eq:uYbpma}), we have
\begin{align}
\underline{Y}_{\beta,-,\alpha}(x)&=\left(-\frac{1}{6}\alpha-\frac{1}{6}\alpha\beta+\frac{1}{6}\beta-\frac{1}{6}\right)e^{-3x}+\left(\frac{2}{3}\alpha+\frac{2}{3}\alpha\beta-\frac{2}{3}\beta+\frac{14}{3}\right), \label{eq:uYb-a}\\
\underline{Y}_{\beta,+,\alpha}(x)&=\left(\frac{1}{6}\alpha+\frac{1}{6}\alpha\beta-\frac{1}{6}\beta-\frac{1}{2}\right)e^{-3x}+\left(-\frac{2}{3}\alpha-\frac{2}{3}\alpha\beta+\frac{2}{3}\beta+6\right).
\end{align}
Similarly, substituting (\ref{eq:o5ba}) and (\ref{eq:o1ba}) for (\ref{eq:oYbpma}), we have
\begin{align}
\overline{Y}_{\beta,-,\alpha}(x)&=\left(-\frac{1}{2}\alpha+\frac{1}{6}\alpha\beta-\frac{1}{6}\beta+\frac{1}{6}\right)e^{-3x}+\left(2\alpha-\frac{2}{3}\alpha\beta+\frac{2}{3}\beta+\frac{10}{3}\right), \\
\overline{Y}_{\beta,+,\alpha}(x)&=\left(\frac{1}{2}\alpha-\frac{1}{6}\alpha\beta+\frac{1}{6}\beta-\frac{5}{6}\right)e^{-3x}+\left(-2\alpha+\frac{2}{3}\alpha\beta-\frac{2}{3}\beta+\frac{22}{3}\right). \label{eq:oYb+a}
\end{align}
Hence, the desired solution is composed of the $(\alpha,\beta)$-cut set
\begin{align*}
[{\cal Y}(x)]_{\beta}^{\alpha}=\left\langle\left[\underline{Y}_{\beta,-,\alpha}(x),\underline{Y}_{\beta,+,\alpha}(x)\right],\ \left[\overline{Y}_{\beta,-,\alpha}(x),\overline{Y}_{\beta,+,\alpha}(x)\right]\right\rangle
\end{align*}
with (\ref{eq:uYb-a})-(\ref{eq:oYb+a}).

The same result is obtained by solving the given problem in the (1,2)-case.
({\tt A.C.})
\end{sol}

\begin{figure}[h]
\begin{center}
\includegraphics[width=5cm]{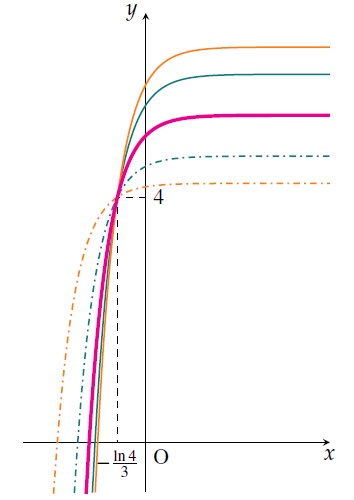}
\caption{The crisp and fuzzy solution of Problem \ref{prob:P1} if $(\alpha,\beta)=(1/3,1/2)$}
\label{fig:P1a1/3b1/2}
\end{center}
\end{figure}

{\small
\begin{rem}
Letting $\alpha=\beta=1$ in (\ref{eq:uYb-a})-(\ref{eq:oYb+a}), we have the crisp solution 
\[
\textcolor{magenta}{y(x)=-\frac{1}{3}e^{-3x}+\frac{16}{3}}
\]
to the crisp IVP:
\[
y''(x)+3y'(x)=0,\quad y(0)=5,\quad y'(0)=1.
\]
We can thus find that type-2 (or type-1) fuzzy differential equation theory is an extension to crisp one.
Also, the solution if $(\alpha,\beta)=(1/3,1/2)$ with the parametric forms
\begin{align*}
\textcolor{teal}{[\underline{Y}_{1/2}(x)]_{1/3}}&\textcolor{teal}{=\left[-\frac{1}{6}e^{-3x}+\frac{14}{3},\ -\frac{1}{2}e^{-3x}+6\right]}, \\
\textcolor{orange}{[\overline{Y}_{1/2}(x)]_{1/3}}&\textcolor{orange}{=\left[-\frac{1}{18}e^{-3x}+\frac{38}{9},\ -\frac{11}{18}e^{-3x}+\frac{58}{9}\right]}
\end{align*}
is as shown in Figure \ref{fig:P1a1/3b1/2}.
\end{rem}
}

\subsection{Case of Negative Coefficients in the Sense of Hukuhara Differences}

\begin{prob}
\label{prob:P2}
Let ${\cal Y}:[0,1]\to \mathscr{T}^2(\mathbb R)$ be a type-2 fuzzy function.
Then, solve T2FIVPs:
\begin{numcases}
{}
{\rm D}^2{\cal Y}(x)-{\cal Y}(x)=0, \label{eq:P2d2-}\\
{\cal Y}(0)={\it 5}\in \mathscr{QT}^2(\mathbb R), \\
{\rm D}{\cal Y}(0)={\it 1}\in \mathscr{QT}^2(\mathbb R).
\end{numcases}
\end{prob}

\begin{sol}
We consider the $(\alpha,\beta)$-cut set of (\ref{eq:P2d2-}):
\begin{align*}
\left\langle[\underline{Y}^{\dag\dag}_{\beta}(x)]_{\alpha}-[\underline{Y}_{\beta}(x)]_{\alpha},\ [\overline{Y}^{\dag\dag}_{\beta}(x)]_{\alpha}-[\overline{Y}_{\beta}(x)]_{\alpha}\right\rangle=0,
\quad \alpha\in [0,1].
\end{align*}
This can be solved by (1,1) or (2,2)-T1-differentiation.
So, considering (1,1)-case, we have the following two T1FIVPs:
\begin{numcases}
{}
[\underline{Y}^{\dag\dag}_{\beta}(x)]_{\alpha}-[\underline{Y}_{\beta}(x)]_{\alpha}=[\underline{Y}''_{\beta,-,\alpha}(x)-\underline{Y}_{\beta,-,\alpha}(x),\ \underline{Y}''_{\beta,+,\alpha}(x)-\underline{Y}_{\beta,+,\alpha}(x)]=0, \label{eq:P2-1u}\\
\underline{Y}_{\beta}(0)=\underline{5}_{\beta}\in \mathscr{T}^1(\mathbb R), \\
\underline{Y}_{\beta}^{\dag}(0)=\underline{1}_{\beta}\in \mathscr{T}^1(\mathbb R), \label{eq:P2uYbd0}
\end{numcases}
and
\begin{numcases}
{}
[\overline{Y}^{\dag\dag}_{\beta}(x)]_{\alpha}-[\overline{Y}_{\beta}(x)]_{\alpha}=[\overline{Y}''_{\beta,-,\alpha}(x)-\overline{Y}_{\beta,-,\alpha}(x),\ \overline{Y}''_{\beta,+,\alpha}(x)-\overline{Y}_{\beta,+,\alpha}(x)]=0, \label{eq:dd2H-1a}\\
\overline{Y}_{\beta}(0)=\overline{5}_{\beta}\in \mathscr{T}^1(\mathbb R), \\
\overline{Y}_{\beta}^{\dag}(0)=\overline{1}_{\beta}\in \mathscr{T}^1(\mathbb R). \label{eq:P2oYbd0}
\end{numcases}
That is, rewriting problems (\ref{eq:P2-1u})-(\ref{eq:P2uYbd0}) and (\ref{eq:dd2H-1a})-(\ref{eq:P2oYbd0}), we have
\begin{align}
\underline{Y}_{\beta,-,\alpha}(x)&=\underline{C}_{-}e^{-x}+\underline{D}_{-}e^{x}, \label{eq:P2first}\\
\underline{Y}_{\beta,-,\alpha}(0)&=\frac{1}{2}\alpha+\frac{1}{2}\alpha\beta-\frac{1}{2}\beta+\frac{9}{2}, \\
\underline{Y}_{\beta,-,\alpha}^{\dag}(0)&=\frac{1}{2}\alpha+\frac{1}{2}\alpha\beta-\frac{1}{2}\beta+\frac{1}{2};
\end{align}
\vspace{-7mm}
\begin{align}
\underline{Y}_{\beta,+,\alpha}(x)&=\underline{C}_{+}e^{-x}+\underline{D}_{+}e^{x}, \\
\underline{Y}_{\beta,+,\alpha}(0)&=-\frac{1}{2}\alpha-\frac{1}{2}\alpha\beta+\frac{1}{2}\beta+\frac{11}{2}, \\
\underline{Y}_{\beta,+,\alpha}^{\dag}(0)&=-\frac{1}{2}\alpha-\frac{1}{2}\alpha\beta+\frac{1}{2}\beta+\frac{3}{2};
\end{align}
\vspace{-7mm}
\begin{align}
\overline{Y}_{\beta,-,\alpha}(x)&=\overline{C}_{-}e^{-x}+\overline{D}_{-}e^{x}, \\
\overline{Y}_{\beta,-,\alpha}(0)&=\frac{3}{2}\alpha-\frac{1}{2}\alpha\beta+\frac{1}{2}\beta+\frac{7}{2}, \\
\overline{Y}_{\beta,-,\alpha}^{\dag}(0)&=\frac{3}{2}\alpha-\frac{1}{2}\alpha\beta+\frac{1}{2}\beta-\frac{1}{2};
\end{align}
\vspace{-7mm}
\begin{align}
\overline{Y}_{\beta,+,\alpha}(x)&=\overline{C}_{+}e^{-x}+\overline{D}_{+}e^{x}, \\
\overline{Y}_{\beta,+,\alpha}(0)&=-\frac{3}{2}\alpha+\frac{1}{2}\alpha\beta-\frac{1}{2}\beta+\frac{13}{2}, \\
\overline{Y}_{\beta,+,\alpha}^{\dag}(0)&=-\frac{3}{2}\alpha+\frac{1}{2}\alpha\beta-\frac{1}{2}\beta+\frac{5}{2}. \label{eq:P2last}
\end{align}
We can obtain $\underline{Y}_{\beta,\pm,\alpha}(x)$ and $\overline{Y}_{\beta,\pm,\alpha}(x)$ by solving these:
\begin{align}
[\underline{Y}_{\beta}(x)]_{\alpha}&=[\underline{Y}_{\beta,-,\alpha}(x),\ \underline{Y}_{\beta,+,\alpha}(x)] \nonumber\\
&=\left[2e^{-x}+\frac{1}{2}(\alpha+\alpha\beta-\beta+5)e^x,\ 2e^{-x}+\frac{1}{2}(-\alpha-\alpha\beta+\beta+7)e^x\right], \label{eq:P2[uYb]a}\\
[\overline{Y}_{\beta}(x)]_{\alpha}&=[\overline{Y}_{\beta,-,\alpha}(x),\ \overline{Y}_{\beta,+,\alpha}(x)] \nonumber\\
&=\left[2e^{-x}+\frac{1}{2}(3\alpha-\alpha\beta+\beta+3)e^x,\ 2e^{-x}+\frac{1}{2}(-3\alpha+\alpha\beta-\beta+9)e^x\right]. \label{eq:P2[oYb]a}
\end{align}
Hence, the desired solution is composed of the $(\alpha,\beta)$-cut
\begin{align*}
[{\cal Y}(x)]_{\beta}^{\alpha}=\left\langle\left[\underline{Y}_{\beta,-,\alpha}(x),\underline{Y}_{\beta,+,\alpha}(x)\right],\ \left[\overline{Y}_{\beta,-,\alpha}(x),\overline{Y}_{\beta,+,\alpha}(x)\right]\right\rangle.
\end{align*}
with (\ref{eq:P2[uYb]a}) and (\ref{eq:P2[oYb]a}).

The same result is obtained by solving the given problem in the (2,2)-case.
({\tt A.C.})
\end{sol}

\begin{figure}[h]
\begin{center}
\includegraphics[width=4cm]{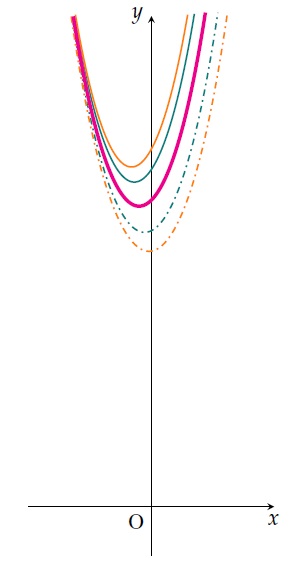}
\caption{The crisp and fuzzy solution of Problem \ref{prob:P2} if $(\alpha,\beta)=(1/3,1/2)$}
\label{fig:P2a1/3b1/2}
\end{center}
\end{figure}

{\small
\begin{rem}
Letting $\alpha=\beta=1$ in (\ref{eq:P2[uYb]a}) and (\ref{eq:P2[oYb]a}), we have the crisp solution 
\[
\textcolor{magenta}{y(x)=2e^{-x}+3e^x}
\]
to the crisp IVP:
\[
y''(x)-y'(x)=0,\quad y(0)=5,\quad y'(0)=1.
\]
We can thus find that type-2 (or type-1) fuzzy differential equation theory is an extension to crisp one.
Also, the solution if $(\alpha,\beta)=(1/3,1/2)$ with the parametric forms
\begin{align*}
\textcolor{teal}{[\underline{Y}_{1/2}(x)]_{1/3}}&\textcolor{teal}{=\left[2e^{-x}+\frac{5}{2}e^x,\ 2e^{-x}+\frac{7}{2}e^x\right]}, \\
\textcolor{orange}{[\overline{Y}_{1/2}(x)]_{1/3}}&\textcolor{orange}{=\left[2e^{-x}+\frac{13}{6}e^x,\ 2e^{-x}+\frac{23}{6}e^x\right]}
\end{align*}
is as shown in Figure \ref{fig:P2a1/3b1/2}.
\end{rem}
}

\subsection{Case of Negative Coefficients in the Usual Sense}

\begin{prob}
\label{prob:P3}
Let ${\cal Y}:[0,1]\to \mathscr{T}^2(\mathbb R)$ be a type-2 fuzzy function.
Then, solve T2FIVPs:
\begin{numcases}
{}
{\rm D}^2{\cal Y}(x)+(-1){\cal Y}(x)=0, \label{eq:d2B3}\\
{\cal Y}(0)={\it 5}\in \mathscr{QT}^2(\mathbb R), \\
{\rm D}{\cal Y}(0)={\it 1}\in \mathscr{QT}^2(\mathbb R).
\end{numcases}
\end{prob}

\begin{sol}
We consider the $(\alpha,\beta)$-cut set of (\ref{eq:d2B3}):
\begin{align*}
\left\langle[\underline{Y}^{\dag\dag}_{\beta}(x)]_{\alpha}+(-1)[\underline{Y}_{\beta}(x)]_{\alpha},\ [\overline{Y}^{\dag\dag}_{\beta}(x)]_{\alpha}+(-1)[\overline{Y}_{\beta}(x)]_{\alpha}\right\rangle=0,
\quad \alpha\in [0,1].
\end{align*}
This can be solved by (1,2) or (2,1)-T1-differentiation.
So, considering (1,2)-case, we have the following two T1FIVPs:
\begin{numcases}
{}
[\underline{Y}^{\dag\dag}_{\beta}(x)]_{\alpha}+(-1)[\underline{Y}_{\beta}(x)]_{\alpha}=[\underline{Y}''_{\beta,+,\alpha}(x)-\underline{Y}_{\beta,+,\alpha}(x),\ \underline{Y}''_{\beta,-,\alpha}(x)-\underline{Y}_{\beta,-,\alpha}(x)]=0, \label{eq:P3-1u}\\
\underline{Y}_{\beta}(0)=\underline{5}_{\beta}\in \mathscr{T}^1(\mathbb R), \\
\underline{Y}_{\beta}^{\dag}(0)=\underline{1}_{\beta}\in \mathscr{T}^1(\mathbb R), \label{eq:P3uYbd0}
\end{numcases}
and
\begin{numcases}
{}
[\overline{Y}^{\dag\dag}_{\beta}(x)]_{\alpha}+(-1)[\overline{Y}_{\beta}(x)]_{\alpha}=[\overline{Y}''_{\beta,+,\alpha}(x)-\overline{Y}_{\beta,+,\alpha}(x),\ \overline{Y}''_{\beta,-,\alpha}(x)-\overline{Y}_{\beta,-,\alpha}(x)]=0, \label{eq:dd2-1a}\\
\overline{Y}_{\beta}(0)=\overline{5}_{\beta}\in \mathscr{T}^1(\mathbb R), \\
\overline{Y}_{\beta}^{\dag}(0)=\overline{1}_{\beta}\in \mathscr{T}^1(\mathbb R). \label{eq:P3oYbd0}
\end{numcases}
We thus gain the same lower and upper equations (\ref{eq:P2first})-(\ref{eq:P2last}) as Problem \ref{prob:P2}.
By solving problems (\ref{eq:P3-1u})-(\ref{eq:P3uYbd0}) and (\ref{eq:dd2-1a})-(\ref{eq:P3oYbd0}), we can obtain $\underline{Y}_{\beta,\pm,\alpha}(x)$ and $\overline{Y}_{\beta,\pm,\alpha}(x)$ respectively.
Hence, the desired solution is the same as \ref{prob:P2} and is composed of the $(\alpha,\beta)$-cut
\begin{align*}
[{\cal Y}(x)]_{\beta}^{\alpha}=\left\langle\left[\underline{Y}_{\beta,-,\alpha}(x),\underline{Y}_{\beta,+,\alpha}(x)\right],\ \left[\overline{Y}_{\beta,-,\alpha}(x),\overline{Y}_{\beta,+,\alpha}(x)\right]\right\rangle.
\end{align*}
with (\ref{eq:P2[uYb]a}) and (\ref{eq:P2[oYb]a}).

The same result is obtained by solving the given problem in the (2,1)-case.
({\tt A.C.})
\end{sol}

{\small
\begin{rem}
Problem \ref{prob:P2} and \ref{prob:P3} imply that Hukuhara differentiation solves the problem of how the negative coefficients of T1/T2FDEs should be treated.
As a result, the negative coefficient may be unified into the Hukuhara difference.
\end{rem}
}

\section{Conclusion}
We have obtained some theorems on type-2 fuzzy calculus and concrete solution methods of T2FIVPs for second-order T2FDEs with constant coefficients in this paper.
T2FDEs with fuzzy coefficients can actually be assumed, but the solution method is exactly the same, although it becomes more complicated than T2FDEs with constant coefficients.

Type-2 fuzzy theory is almost equivalent to considering two (i.e. left and right) type-1 fuzzy sets in the end.
At first glance, it sounds like type-2-discussions are not necessary and type-1-discussions are sufficient.
For example, we can actually discuss the problem of measurement by a highly experienced expert and an inexperienced student, that is explained in Introduction, by comparing their type-1 membership functions.
However, it is possible to see and catch the difference between the two as integrated information by virtue of concepts of the foot-print and principle sets.
The same is true when the foot-print set (resp. principle set) is regarded as the old (resp. present) experiment in the case that parameters of second-order differential equations are determined by experimental rules.
We believe that here is the importance and usefulness of type-2 fuzzy theory or T2FDEs.

When we want to judge whether a certain value is close to $3$, for example, the type-2 fuzzy number is used if the grade of $3$ is `about $0.7$'.
In this paper, we set Definition \ref{df:T2FS} to describe that situation.
There are other definitions of a type-2 fuzzy set.
For example, we can also define a type-2 fuzzy set ${\cal A}$ by
\begin{align*}
&{\cal A}:=\{((x,\mu_0(x)),\ \mu_{{\cal A}}(\mu_0(x),u)):x\in X\}, \\
&\mu_0:X\to [0,1], \\
&\mu_{{\cal A}}:X\times R(\mu_0)\to [0,1]
\end{align*}
as the definition that differs our Definition \ref{df:T2FS}.
Here $R(\mu_0)$ denotes the range of $\mu_0$ called the primary membership function.
This definition considers the membership function $\mu_{{\cal A}}$ to be a two-variable function that implicitly contains the type-1 membership function $\mu_0$.
In particular, a `triangular' type-2 fuzzy number ${\cal U}$ can be also defined as 
\begin{align*}
\mu_{{\cal U}}(\mu_0(x),u):=
\begin{cases}
\displaystyle 1-\frac{|u-\mu_0(x)|}{\min\{\mu_0(x),1-\mu_0(x)\}}, & \mu_0(x)\neq \{0,1\}; \vspace{1mm}\\
\begin{cases}
1, & u=\mu_0(x); \\
0, & {\rm otherwise};
\end{cases} 
& \mu_0(x)=\{0,1\}.
\end{cases}
\end{align*} 
Then, primary $\mu_0$ should be decided subjectively by us.
For example, if we set $\mu_0(x)=\langle\!\langle 1,2,3\rangle\!\rangle$, the graph of ${\cal U}$ is as shown Figure \ref{fig:TTFN}.
\begin{figure}[h]
\begin{center}
\includegraphics[width=6cm]{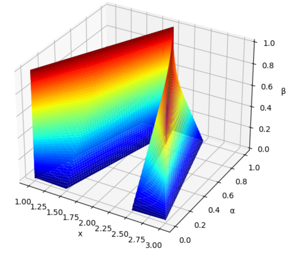}
\caption{`Triangular' type-2 fuzzy number with $\mu_0(x)=\langle\!\langle 1,2,3\rangle\!\rangle$}
\label{fig:TTFN}
\end{center}
\end{figure}
This definition fits our senses better, but the calculations will be complicated.
In fact, it can be seen that the sides are curved unlike perfect quasi-type-2 fuzzy numbers.
We would like to study the differential equations for type-2 fuzzy numbers in the case of this definition as a future task.

\appendix
\section{Appendix}
In this appendix, the notations, definitions and theorems on type-1 fuzzy numbers are written.

\subsection{Type-1 Fuzzy Numbers}

\begin{df}
A fuzzy set $u:\mathbb R\to [0,1]$ is the type-1 fuzzy number if and only if 
\begin{itemize}
\item[i)] $u$ is normal, that is, there exists an $x_0\in \mathbb R$ such that $u(x_0)=1$;
\item[ii)] $u$ is fuzzy convex, that is, $u(tx+(1-t)y)\ge \min\{u(x),u(y)\}$ for any $x,y\in \mathbb R$ and $t\in [0,1]$;
\item[iii)] $u$ is upper semi-continuous;
\item[iv)] the support ${\rm supp}(u):={\rm cl}(\{x\in \mathbb R:u(x)>0\})$ of $u$ is compact, where ${\rm cl}(S)$ stands for the closure of the crisp set $S$.
\end{itemize}
\end{df}

It is well known in the study of fuzzy theory so far that the argument of fuzzy sets can be reduced to that of the level cut sets.
The $\alpha$-cut set $[A]_{\alpha}$ of a type-1 fuzzy set $A$ on $X$ is defined by
\begin{align*}
[A]_{\alpha}&:=\{x\in X:A(x)\ge \alpha\},\quad \alpha\in (0,1]; \\
[A]_0&:={\rm supp}(u), \hspace{2.1cm} \alpha=0,
\end{align*}
and these level cut sets make up the original type-1 fuzzy set $A$:
\begin{align*}
A=\bigcup_{\alpha\in [0,1]}\alpha[A]_{\alpha}
\end{align*}
where $\alpha[A]_{\alpha}:X\to \{0,\alpha\}$ is a type-1 fuzzy set. 
Thus, it is sufficient to consider and argue $\alpha$-cut sets for most problems.
In particular, the type-1 fuzzy number is the fuzzy set whose arbitrary $\alpha$-cut set is a bounded and closed interval, so we put the following notation.

\begin{df}
We denote the $\alpha$-cut set of a type-1 fuzzy number $u:\mathbb R\to [0,1]$ by
\begin{align}
\label{eq:para}
[u]_{\alpha}=[u_{-,\alpha},u_{+,\alpha}]
\end{align}
for any $\alpha\in [0,1]$.
(\ref{eq:para}) is called the parametric form of $u$.
\end{df}

{\small
\begin{rem}
We write $0$ for crisp zero in this paper and its $\alpha$-cut set is represented  as the closed interval $[0,0]$.
In the same way, we regard a crisp real number $r$ as the closed interval $[r,r]$ and bring crisp numbers into the group of fuzzy numbers.
\end{rem}
}

\begin{rem}
Each end of $[u]_{\alpha}=[u_{-,\alpha},u_{+,\alpha}]$ should satisfy the followings:
\begin{itemize}
\item[1)] $u_{-,\alpha}$ is bounded, monotone increasing, left-continuous with respect to $\alpha\in (0,1]$ and right-continuous on $\alpha=0$;
\item[2)] $u_{+,\alpha}$ is bounded, monotone decreasing, left-continuous with respect to $\alpha\in (0,1]$ and right-continuous on $\alpha=0$;
\item[3)] $u_{-,\alpha}\le u_{+,\alpha}$ for any $\alpha\in [0,1]$.
\end{itemize}
In short, $[u]_{\alpha}$ that does not satisfy any of the above property cannot be called the fuzzy number.
\end{rem}

We hereafter omit the description `$\alpha\in [0,1]$' when we argue $\alpha$-cut sets or parametric forms.

\begin{df}
Let $u,v$ be type-1 fuzzy numbers on $\mathbb R$.
We denote the $\alpha$-cut sets of $u,v$ by 
\begin{align}
\label{eq:uvalpha}
[u]_{\alpha}=[u_{-,\alpha},u_{+,\alpha}],\quad 
[v]_{\alpha}=[v_{-,\alpha},v_{+,\alpha}]
\end{align}
respectively.
Then, $u=v$ if and only if $[u]_{\alpha}=[v]_{\alpha}$, i.e.
\begin{align*}
u_{-,\alpha}=v_{-,\alpha}\quad {\rm and}\quad 
u_{+,\alpha}=v_{+,\alpha}
\end{align*}
for any $\alpha\in [0,1]$.
\end{df}

\begin{df}
\label{df:+kop1}
Let $u,v$ be type-1 fuzzy numbers on $\mathbb R$ and $k\in \mathbb R$.
We denote the $\alpha$-cut sets of them by (\ref{eq:uvalpha}).
The sum $u+v$ of $u$ and $v$ is defined by
\begin{align*}
[u+v]_{\alpha}:=[u]_{\alpha}+[v]_{\alpha},
\end{align*}
i.e.
\begin{align*}
[u_{-,\alpha},u_{+,\alpha}]+[v_{-,\alpha},v_{+,\alpha}]
=[u_{-,\alpha}+v_{-,\alpha},\ u_{+,\alpha}+v_{+,\alpha}].
\end{align*}
Moreover, the scalar multiple $ku$ of $u$ is defined by
\begin{align*}
[ku]_{\alpha}:=k[u]_{\alpha}
=
\begin{cases}
[ku_{-,\alpha},ku_{+,\alpha}], & k\ge 0; \vspace{1mm}\\
[ku_{+,\alpha},ku_{-,\alpha}], & k<0.
\end{cases}
\end{align*}
\end{df}

\begin{df}
Let $u,v$ be type-1 fuzzy numbers on $\mathbb R$.
We denote the $\alpha$-cut sets of them by (\ref{eq:uvalpha}).
The product $uv$ of $u$ and $v$ is defined by
\begin{align*}
[uv]_{\alpha}=[u]_{\alpha}[v]_{\alpha}
=\left[\min_{i,j\in \{-,+\}}u_{i,\alpha}v_{j,\alpha},\ \max_{i,j\in \{-,+\}}u_{i,\alpha}v_{j,\alpha}\right]
\end{align*}
for any $\alpha\in [0,1]$.
\end{df}

\begin{df}
Let $u,v$ be type-1 fuzzy numbers on $\mathbb R$.
We denote the $\alpha$-cut sets of them by (\ref{eq:uvalpha}).
Then, the order relationship between them, $u\le v$, is defined as
\begin{align*}
[u]_{\alpha}\le [v]_{\alpha},
\quad {\rm i.e.}\quad
u_{-,\alpha}\le v_{-,\alpha}\ {\rm and}\ u_{+,\alpha}\le v_{+,\alpha}
\end{align*}
for any $\alpha \in [0,1]$.
In particular, the non-negativity (resp. positivity) of a type-1 fuzzy number $u$,  $u\ge 0$ (resp. $u>0$), is defined by
\begin{align*}
u_{-,\alpha}\ge 0\quad ({\rm resp.}\ u_{-,\alpha}>0)
\end{align*}
for any $\alpha \in [0,1]$.
\end{df}

\begin{df}
Let $u,v$ be type-1 fuzzy numbers on $\mathbb R$.
We denote the $\alpha$-cut sets of them by (\ref{eq:uvalpha}).
The fuzzy Hausdorff distance $d_{{\rm H}}$ of $u$ and $v$ is defined by
\begin{align*}
d_{{\rm H}}(u,v):=\sup_{\alpha\in [0,1]}\max\{|u_{-,\alpha}-v_{-,\alpha}|,|u_{+,\alpha}-v_{+,\alpha}|\}.
\end{align*}
We write $\mathscr{T}^1(\mathbb R)$ for the type-1 fuzzy number space equipped with the $d_{{\rm H}}$-topology.
\end{df}

\subsection{Type-1 Fuzzy Number-valued Functions \label{app:T1FNVF}}
Let $I$ be an interval which is a proper subset of $\mathbb R$ or let $I=\mathbb R$.
We can consider the type-1 fuzzy function $F:\mathscr{T}^1(I)\to \mathscr{T}^1(\mathbb R)$ by using Zadeh's extension principle for a crisp function $f:I\to \mathbb R$.
We set $\mathscr{T}^1(I)=I$ in this paper, that is, we consider crisp-variable type-1 fuzzy number-valued functions exclusively.
The $\alpha$-cut set of $F:I\to \mathscr{T}^1(\mathbb R)$ is represented by
\begin{align*}
[F(x)]_{\alpha}:=[F_{-,\alpha}(x),\ F_{+,\alpha}(x)]
\end{align*}
for all $x\in I$ and any $\alpha\in [0,1]$.

We define the difference of type-1 fuzzy numbers as the following sense, so as to consider the difference quotient of $F$.

\begin{df}
Let $u,v\in \mathscr{T}^1(\mathbb R)$.
If there exists some $w\in \mathscr{T}^1(\mathbb R)$ such that $u=v+w$, we write $w=u-v$ and call it the T1-Hukuhara difference of $u$ and $v$.
\end{df}

{\small
\begin{rem}
For any $u\in \mathscr{T}^1(\mathbb R)$, $-u$ stands for $0-u$, i.e.
\begin{align*}
[-u]_{\alpha}=[-u_{-,\alpha},-u_{+,\alpha}],
\end{align*}
whereas
\begin{align*}
[(-1)u]_{\alpha}=[-u_{+,\alpha},-u_{-,\alpha}].
\end{align*}
We should remark that, in general, 
\begin{align*}
u+(-1)v\neq u-v
\end{align*}
for any $u,v\in \mathscr{T}^1(\mathbb R)$.
\end{rem}
}

We adopt, in this paper, dagger ${\dag}$ and double dagger ${\ddag}$ to denote fuzzy derivatives in the sense of Hukuhara.
We use prime $'$ as the crisp derivative notation.

\begin{df}
Let $F:I\to \mathscr{T}^1(\mathbb R)$ and $h>0$ be a crisp number.
$F$ is T1-differentiable in the first form at some $x_0\in I$ if and only if there exist $F(x_0+h)-F(x_0)$ and $F(x_0)-F(x_0-h)$ satisfying that the fuzzy limit
\begin{align*}
F^\dag(x_0)
:=\lim_{h\downarrow 0}\frac{F(x_0+h)-F(x_0)}{h}
=\lim_{h\downarrow 0}\frac{F(x_0)-F(x_0-h)}{h}
\end{align*}
exists.
Moreover, $F$ is T1-differentiable in the second form at some $x_0\in I$ if and only if there exist $F(x_0)-F(x_0+h)$ and $F(x_0-h)-F(x_0)$ satisfying that the fuzzy limit
\begin{align*}
F^\ddag(x_0)
:=\lim_{h\downarrow 0}\frac{F(x_0)-F(x_0+h)}{-h}
=\lim_{h\downarrow 0}\frac{F(x_0-h)-F(x_0)}{-h}
\end{align*}
exists.
Here the above differences (resp. limits) are due to the meaning of T1-Hukuhara (resp. $d_{{\rm H}}$).
If $F$ is T1-differentiable in both senses at any $x\in I$, $F^\dag$ and $F^\ddag$ is called the (1)-T1-derivative and (2)-T1-derivative of $F$, respectively.
\end{df}

{\small
\begin{rem}
\label{rem:thirdfourth}
In addition to the above two forms, it is possible to consider the following two forms, that is,
\begin{itemize}
\item[i)] the third form:
\begin{align}
\label{eq:thirdf}
\lim_{h\downarrow 0}\frac{F(x_0+h)-F(x_0)}{h}
=\lim_{h\downarrow 0}\frac{F(x_0-h)-F(x_0)}{-h},
\end{align}
\item[ii)] the fourth form:
\begin{align}
\label{eq:fourthf}
\lim_{h\downarrow 0}\frac{F(x_0)-F(x_0+h)}{-h}
=\lim_{h\downarrow 0}\frac{F(x_0)-F(x_0-h)}{h}.
\end{align}
\end{itemize}
However, it is known (\cite{BG}, Theorem 7) that (\ref{eq:thirdf})  becomes the crisp number if there exist $F(x_0+h)-F(x_0)$ and $F(x_0-h)-F(x_0)$.
(\ref{eq:fourthf}) is also so if $F(x_0)-F(x_0+h)$ and $F(x_0)-F(x_0-h)$.
We shall thus ignore the third and fourth forms.
\end{rem}
}

{\small
\begin{rem}
We mention the validity of our dagger symbols $\dag,\ddag$ meaning as fuzzy derivatives.
First of all, we can avoid confusion between the crisp derivative and the fuzzy derivative by using $\dag$ and $\ddag$ (see e.g. Theorems \ref{thm:Ffda}, \ref{thm:Ff2da}, \ref{thm:cftimesff} and \ref{thm:compo} later).
Secondly, if we want to the number of the differential order to 2 or less, we can use $\dag,\ddag$ in the same sense as prime $'$ to clarify what kind of fuzzy derivative it is (in particular, see Theorems \ref{thm:MNpara}, \ref{thm:FT2sdba}, \ref{thm:F+GsHdd} and \ref{thm:F-GT2d} later).
Moreover, the acronym for a letter is often used to represent a mathematical concept or quantity, such as $\Delta$ ({\bf D}elta) for differences or $D$ for derivatives.
From this point of view, `dagger' has an appropriate acronym {\bf D} to represent derivatives.
\end{rem}
}

By using this definition repeatedly, we find that the $n$th-order T1-derivative of $F$ has $2^n$ forms.
For example, when applied to fuzzy differential equations, we need to choose the most appropriate solution from these $2^n$ solutions.

\begin{thm}
\label{thm:Ffda}
Let $F:I\to \mathscr{T}^1(\mathbb R)$ be T1-differentiable on $I$.
Then, the parametric forms of its T1-derivatives are given by
\begin{itemize}
\item[1)] the first parametric form:
\begin{align*}
[F^\dag(x)]_{\alpha}=[F_{-,\alpha}'(x),F_{+,\alpha}'(x)],
\end{align*}
\item[2)] the second parametric form:
\begin{align*}
[F^\ddag(x)]_{\alpha}=[F_{+,\alpha}'(x),F_{-,\alpha}'(x)].
\end{align*}
\end{itemize}
\end{thm}

\begin{thm}[\cite{K}, Theorem 5.2; \cite{CCRF}, Theorem 5]
\label{thm:Ff2da}
Let $F:I\to \mathscr{T}^1(\mathbb R)$ be second-order T1-differentiable on $I$.
Then, the parametric forms of the second-order T1-derivatives are given by
\begin{itemize}
\item[1)] the first and first parametric form:
\begin{align*}
[F^{\dag\dag}(x)]_{\alpha}=[F_{-,\alpha}''(x),F_{+,\alpha}''(x)],
\end{align*}
\item[2)] the first and second parametric form:
\begin{align*}
[F^{\dag\ddag}(x)]_{\alpha}=[F_{+,\alpha}''(x),F_{-,\alpha}''(x)],
\end{align*}
\item[3)] the second and first parametric form:
\begin{align*}
[F^{\ddag\dag}(x)]_{\alpha}=[F_{+,\alpha}''(x),F_{-,\alpha}''(x)],
\end{align*}
\item[4)] the second and second parametric form:
\begin{align*}
[F^{\ddag\ddag}(x)]_{\alpha}=[F_{-,\alpha}''(x),F_{+,\alpha}''(x)].
\end{align*}
\end{itemize}
\end{thm}
\vspace{2mm}

{\small
\begin{center}
{\bf Acknowledgement}
\end{center}

The authors are deeply grateful to Dr. Jiro Inaida for his valuable and heartfelt advice on fuzzy number theory.
}

\vspace{5mm}

{\small
\noindent
{\bf The Authors}:
\vspace{2mm}

\noindent
{\sc Norihiro Someyama};

\quad Completed Ph.D. program without a Ph.D. degree, Major in Mathematics and Mathematical Physics
\vspace{1mm}

\noindent
{\sc Hiroaki Uesu};\quad Ph.D., Major in Mathematics and Information Theory
\vspace{1mm}

\noindent
{\sc Kimiaki Shinkai};\quad Ph.D., Major in Mathematics and Mathematical Education
\vspace{1mm}

\noindent
{\sc Shuya Kanagawa};\quad Ph.D., Major in Mathematics and Statistics
}
\end{document}